\documentclass[british]{article}
\usepackage{ae,aecompl}
\usepackage[T1]{fontenc}
\usepackage[latin9]{inputenc}
\usepackage{parskip}
\usepackage{mathtools}
\usepackage{amsmath}
\usepackage{amsthm}
\usepackage{amssymb}
\usepackage{geometry}
\makeatletter

\providecommand{\tabularnewline}{\\}

\theoremstyle{plain}
\newtheorem{thm}{\protect\theoremname}
\theoremstyle{definition}
\newtheorem{defn}[thm]{\protect\definitionname}
\theoremstyle{remark}
\newtheorem{rem}[thm]{\protect\remarkname}
\theoremstyle{plain}
\newtheorem{prop}[thm]{\protect\propositionname}
\theoremstyle{plain}
\newtheorem{lem}[thm]{\protect\lemmaname}

\usepackage{capt-of}
\usepackage{bbm}
\date{}

\theoremstyle{definition}
\newtheorem{notation}[thm]{Notation}
\makeatother

\usepackage{babel}
\providecommand{\definitionname}{Definition}
\providecommand{\lemmaname}{Lemma}
\providecommand{\propositionname}{Proposition}
\providecommand{\remarkname}{Remark}
\providecommand{\theoremname}{Theorem}

\begin{document}
\global\long\def\IN{\mathbb{N}}%
\global\long\def\II{\mathbbm{1}}%
\global\long\def\IZ{\mathbb{Z}}%
\global\long\def\IQ{\mathbb{Q}}%
\global\long\def\IR{\mathbb{R}}%
\global\long\def\IC{\mathbb{C}}%
\global\long\def\IP{\mathbb{P}}%
\global\long\def\IE{\mathbb{E}}%
\global\long\def\IV{\mathbb{V}}%

\title{Reconstructing the Probability Measure of a Curie-Weiss Model Observing
the Realisations of a Subset of Spins}
\author{Miguel Ballesteros\thanks{IIMAS-UNAM, Mexico City, Mexico}, Ivan
Naumkin\footnotemark[1], and Gabor Toth\footnotemark[1]
\thanks{Corresponding author, e-mail: gabor.toth@iimas.unam.mx}}
\maketitle
\begin{abstract}
\noindent We study the problem of reconstructing the probability
measure of the Curie-Weiss model from a sample of the voting behaviour
of a subset of the population. While originally used to study phase
transitions in statistical mechanics, the Curie-Weiss or mean-field
model has been applied to study phenomena, where many agents interact
with each other. It is useful to measure the degree of social cohesion
in social groups, which manifests in the way the members of the group
influence each others' decisions. In practice, statisticians often
only have access to survey data from a representative subset of a
population. As such, it is useful to provide methods to estimate social
cohesion from such data. The estimators we study have some positive
properties, such as consistency, asymptotic normality, and large deviation
principles. The main advantages are that they require only a sample
of votes belonging to a (possibly very small) subset of the population
and have a low computational cost. Due to the wide application of
models such as Curie-Weiss, these estimators are potentially useful
in disciplines such as political science, sociology, automated voting,
and preference aggregation.
\end{abstract}


\section{Introduction}

The study of ferromagnetism has played a central role in the development
of statistical mechanics, offering insights into how large ensembles
of interacting microscopic units can give rise to macroscopic order.
At the heart of this endeavour lie mathematical models such as the
Ising and Curie-Weiss models, whose simplicity belies their deep mathematical
structure and broad applicability. In the Ising model, introduced
by Lenz and Ising in the 1920s \cite{Ising1925}, binary-valued spins
are arranged on the lattice $\IZ^{d}$ for some fixed dimension $d\in\IN$,
with each spin influenced by its immediate neighbours and external
magnetic fields. Although the model's one-dimensional case is analytically
tractable and exhibits no phase transition at finite temperature,
higher-dimensional versions reveal complex behaviour, including phase
transitions and criticality, a result first rigorously demonstrated
by Onsager in two dimensions \cite{Onsager1944}.

The Curie-Weiss model, developed slightly earlier by the physicists
Pierre Curie and Pierre Weiss in order to study phase transitions,
adopts a different approach by replacing local interactions with a
global coupling where each spin interacts symmetrically with every
other. This mean-field formulation makes the model mathematically
accessible while preserving key features of collective phenomena,
such as spontaneous magnetisation. While it originated in the context
of statistical physics, the Curie-Weiss model has proven remarkably
versatile, finding applications in fields as diverse as neuroscience
\cite{BouMcDMu2007}, economics (there is a large body of literature
dealing with models of strategic interactions of economic agents using
a mean-field approach as in the Curie-Weiss model, e.g. \cite{JovaRose1988,BD2001}),
queuing theory \cite{BaKaKPRS1992}, sociology \cite{AlBaKhBu2022},
and political science \cite{DieKauKa2023}. The model\textquoteright s
appeal lies in its ability to capture consensus formation, peer influence,
and polarisation through a minimal set of assumptions about agent
interactions.

One fruitful direction of generalisation of the Curie-Weiss model,
motivated by empirical contexts where agents do not form a single
homogeneous population, is to divide the population into multiple,
say $M\in\IN$, groups with potentially differing internal dynamics.
This yields a multi-group Curie-Weiss model, in which the strength
of interaction may vary across groups. A growing body of literature
has investigated this framework, with applications ranging from opinion
dynamics and voting theory \cite{KT2021c} to statistical community
detection \cite{BRS2019,LoweSchu2020,BaMePeTo2023}. We will consider
the case of non-interacting groups, and thus the model has a set of
coupling non-negative coupling constants $\beta_{\lambda}$, where
each group is labelled by some $\lambda\in\mathbb{N}_{M}$\footnote{We will write $\IN_{m}\coloneq\left\{ 1,\ldots,m\right\} $ for any
$m\in\IN$.}. Each coupling constant $\beta_{\lambda}$ regulates the interaction
between agents in group $\lambda$. The multi-group formulation can
accommodate differences in cultural norms or ideologies or social
cohesion within subpopulations.

Previous articles explored the estimation problem of the coupling
parameters from a sample of votes from the entire population. \cite{BaMeSiTo2025}
studies the maximum likelihood estimator which is consistent (the
estimator converges in probability to the true parameter value), asymptotically
normal (the normal fluctuations allowing for the calculation of confidence
intervals), and satisfies a large deviation property (ensuring swift
decay of the probability of a large deviation of the estimator from
the true parameter). Despite these desirable properties, the maximum
likelihood estimator has a drawback which may make it impractical
in many applications\ involving large voter populations, such as
countries: it necessitates the calculation of the so called partition
function which scales exponentially with the size of the population.
For this reason, the article \cite{BalNauTo2025} studies an approximation
of the maximum likelihood estimator valid for large populations. This
estimator can be calculated from a sample of votes under a constant
and low computational cost. It shares the statistical properties of
the maximum likelihood estimator, with the consistency being subject
to the assumption that the population is large enough. While this
estimator is more practical than its maximum likelihood counterpart,
it shares the drawback of being based on a sample of votes of the
entire population.

In this article, we explore two estimators which can be calculated
from a sample of votes of a (potentially very small) subset of the
entire population, thus allowing the reconstruction of the Curie-Weiss
probability measure in many cases which arise in practice, such as
from survey data from a representative subset of a population. One
family of estimators is based on the idea of applying the optimality
condition (\ref{eq:opt}) characterising the maximum likelihood estimator
to the data from a subset of votes. The second family of estimators
is based on observing the empirical pair correlation between votes
and contrasting with the theoretical value for the Curie-Weiss model
given a set of coupling parameters. Both estimators have a low computational
cost similar to the estimator in \cite{BalNauTo2025}, thus combining
the two main properties of a practical estimator in the context of
large populations: low computational cost and only requiring a sample
of votes from a small subset of each group.

Though the definition of our two estimators are very different, it
turns out that for most samples these two estimators will yield very
similar values, and in fact, as the population goes to infinity, they
yield the same estimator (see Theorem \ref{thm:asymp_equiv}). The
main trade-off of calculating these estimators from a subset of votes
is a higher variance of the estimator when compared to using a sample
from the entire population.

A compelling use case of the multi-group Curie-Weiss model is in voting
theory. Here, spins are reinterpreted as binary votes in elections,
and the interaction structure reflects the degree of alignment within
each group of voters. When decisions are aggregated in a two-tier
voting system, such as a federal council composed of representatives
from different regions or member states, the question arises of how
to assign voting weights that appropriately reflect both group size
and internal cohesion. This problem is not merely academic: real-world
institutions such as the Council of the European Union rely on weighted
voting mechanisms to balance the influence of countries with varying
populations and political cultures. Depending on the fairness criterion
underlying the assignation of voting weights corresponding to each
group, a crucial ingredient is a probabilistic voting model which
describes in statistical terms how voters interact with each other.
Thus, a central statistical task is to infer the group-level interaction
strengths from empirical voting data, which then inform the allocation
of fair and representative weights. See \cite[Section 7]{BaMeSiTo2025}
and \cite[Section 5]{BalNauTo2025} for a study of how to estimate
the voting weights. The estimators found in this paper can also be
employed for this purpose.

The broader significance of this work lies in its synthesis of statistical
mechanics, inference theory, and institutional design. It contributes
to a line of research that leverages physical models to shed light
on questions of fairness, aggregation, and representation in complex
systems. In doing so, it offers rigorous tools for both understanding
and designing voting schemes that respond to real-world patterns of
group behaviour. Our exposition is intended to be accessible to a
broad academic audience, combining formal mathematical proofs with
conceptual clarity for readers across disciplines, including mathematics,
statistics, physics, economics, and political science.

The rest of the article is arranged as follows: Section \ref{sec:Curie-Weiss-Model}
defines the Curie-Weiss model. The main results concerning the statistical
properties of the two families of estimators and their large population
asymptotic equivalence are found in Section \ref{sec:main_res}. The
proofs of these results are presented in Section \ref{sec:Proofs},
and an Appendix contains some auxiliary results we employ in our proofs.

\section{\label{sec:Curie-Weiss-Model}The Curie-Weiss Model}

We have a voting population subdivided into $M\in\IN$ groups of size
$N_{\lambda}\in\IN$, $\lambda\in\IN_{M}$ each. The total size of
the population is then $N_{1}+\cdots+N_{M}$, and the votes take values
in the space
\[
\Omega_{N_{1}+\cdots+N_{M}}\coloneq\left\{ -1,1\right\} ^{N_{1}+\cdots+N_{M}}.
\]
Each individual is referenced by two indices $\lambda\in\IN_{M}$
for their group and $i\in\IN_{N_{\lambda}}$ for the individual. The
random variable $X_{\lambda i}$ is the vote of person $i$ in group
$\lambda$. Each vote takes a value $X_{\lambda i}=x_{\lambda i}\in\Omega_{1}$.
The elements $\left(x_{11},\ldots,x_{1N_{1}},\ldots,x_{M1},\ldots,x_{MN_{M}}\right)\in\Omega_{N_{1}+\cdots+N_{M}}$
will be called voting configurations, with each voting configuration
being a complete record of the votes cast by the population on a specific
issue. The voting behaviour is described by the following voting model:
\begin{defn}
\label{def:CWM}Let $N_{\lambda}\in\IN$ and $\beta_{\lambda}\in\IR$,
$\lambda\in\IN_{M}$. We define the vectors $\boldsymbol{N}\coloneq\left(N_{1},\ldots,N_{M}\right)$
and $\boldsymbol{\beta}\coloneq\left(\beta_{1},\ldots,\beta_{M}\right)$.
The\emph{ Curie-Weiss model} (CWM) is defined for all voting configurations\\
$\left(x_{11},\ldots,x_{1N_{1}},\ldots,x_{M1},\ldots,x_{MN_{M}}\right)\in\Omega_{N_{1}+\cdots+N_{M}}$
by
\begin{equation}
\IP_{\boldsymbol{\beta},\boldsymbol{N}}\left(X_{11}=x_{11},\ldots,X_{MN_{M}}=x_{MN_{M}}\right)\coloneq Z_{\boldsymbol{\beta},\boldsymbol{N}}^{-1}\,\exp\left(\frac{1}{2}\sum_{\lambda=1}^{M}\frac{\beta_{\lambda}}{N_{\lambda}}\left(\sum_{i=1}^{N_{\lambda}}x_{\lambda i}\right)^{2}\right),\label{eq:CWM}
\end{equation}
where $Z_{\boldsymbol{\beta},\boldsymbol{N}}$ is called the partition
function which depends on both $\boldsymbol{\beta}$ and $\boldsymbol{N}$.
The constants $\beta_{\lambda}$ are referred to as coupling parameters.
\end{defn}

The original CWM with $M=1$ was conceived as a model of ferromagnetism
with a single coupling parameter $\beta$ which represents the inverse
temperature of the system. The range of values for $\beta$ is thus
usually $\left[0,\infty\right)$. We allow values of $\beta$ outside
this range for technical reasons related to the range of the sample
statistics $\boldsymbol{P}_{\boldsymbol{K}_{\boldsymbol{N}}}$ and
$\boldsymbol{T}_{\boldsymbol{K}_{\boldsymbol{N}}}$ (see Definitions
\ref{def:PKN} and \ref{def:TKN} below).

As a model of voting, the CWM has non-negative coupling parameters
as a reflection of social cohesion within each group. Each group's
coupling parameter $\beta_{\lambda}$ is a measure of how much the
voters influence each other in their decisions, with the influence
being stronger for higher values of $\beta_{\lambda}$. This influence
is stronger for larger $\beta_{\lambda}$. The most probable voting
configurations under the measure defined in (\ref{eq:CWM}) are the
unanimous votes either for or against the proposal. The far more numerous
configurations with roughly equal numbers of votes for and against
are individually of lower probability than configurations with large
majorities. What the size of the typical majority under the model
is depends on the magnitude of the coupling parameters. For non-interacting
groups, each group can be in any of the three regimes of the CWM determined
by the value of the parameter $\beta_{\lambda}\geq0$. The three regimes
are listed in Table \ref{tab:Regimes}.

\begin{table}
\begin{centering}
\begin{tabular}{|c|c|c|}
\hline 
High temperature regime & Critical regime & Low temperature regime\tabularnewline
\hline 
\hline 
$\beta_{\lambda}<1$ & $\beta_{\lambda}=1$  & $\beta_{\lambda}>1$\tabularnewline
\hline 
\end{tabular}
\par\end{centering}
\caption{\label{tab:Regimes}Regimes of each group}

\end{table}

The partition function $Z_{\boldsymbol{\beta},\boldsymbol{N}}$ of
the CWM is a normalising constant given by the sum of the exponentials
for each voting configuration:
\begin{equation}
Z_{\boldsymbol{\beta},\boldsymbol{N}}=\sum_{x\in\Omega_{N_{1}+\cdots+N_{M}}}\exp\left(\frac{1}{2}\sum_{\lambda=1}^{M}\frac{\beta_{\lambda}}{N_{\lambda}}\left(\sum_{i=1}^{N_{\lambda}}x_{\lambda i}\right)^{2}\right).\label{eq:part_fn}
\end{equation}
The group voting margins (i.e.\! the difference between the numbers
of yes and no votes) for each $\lambda\in\IN_{M}$ are defined as
\begin{equation}
S_{\lambda}\coloneq\sum_{i=1}^{N_{\lambda}}X_{\lambda i},\quad\lambda\in\IN_{M}.\label{eq:S_lambda}
\end{equation}
If a function $f:\Omega_{N_{\lambda}}\rightarrow\IR$ depends only
on the votes belonging to group $\lambda$, such as $S_{\lambda}$
above, we will write $\IE_{\beta_{\lambda},N_{\lambda}}f$ for the
expectation $\IE_{\boldsymbol{\beta},\boldsymbol{N}}f$ as it depends
only on the marginal distribution $\IP_{\beta_{\lambda},N_{\lambda}}$
on $\Omega_{N_{\lambda}}$.

\begin{notation}Throughout the article, the symbol $\IE X$ will
stand for the expectation and $\IV X$ for the variance of some random
variable $X$. Capital letters such as $X$ will denote random variables,
while lower case letters such as $x$ will stand for realisations
of the corresponding random variable.\end{notation}

\section{\label{sec:main_res}Estimators of $\boldsymbol{\beta}$ Based on
Subsets of Votes}

The maximum likelihood estimator $\hat{\boldsymbol{\beta}}_{ML}$
of $\boldsymbol{\beta}$ requires a sample $\left(x^{(1)},\ldots,x^{(n)}\right)\in\Omega_{N_{1}+\cdots+N_{M}}^{n}$
and is obtained as the value which maximises the likelihood function
given $\left(x^{(1)},\ldots,x^{(n)}\right)$. This optimality condition
is equivalent to
\begin{equation}
\IE_{\hat{\boldsymbol{\beta}}_{ML}\left(\lambda\right),N_{\lambda}}S_{\lambda}^{2}=\frac{1}{n}\sum_{t=1}^{n}\left(\sum_{i=1}^{N}x_{\lambda i}^{\left(t\right)}\right)^{2}\label{eq:opt}
\end{equation}
for each $\lambda\in\IN_{M}$, where $\hat{\boldsymbol{\beta}}_{ML}\left(\lambda\right)$
is the maximum likelihood estimate for the parameter $\beta_{\lambda}$
given the sample $\left(x^{(1)},\ldots,x^{(n)}\right)$. See \cite[Section 3]{BaMeSiTo2025}
for a derivation of (\ref{eq:opt}) and the definition of the maximum
likelihood estimator for $\boldsymbol{\beta}$.

We will assume we have access to $n\in\IN$ observations of voting
outcomes. However, instead of observing the votes of the entire population
of size $N_{1}+\cdots+N_{M}$, we only observe a subset of votes.

From each group $\lambda\in\IN_{M}$ of size $N_{\lambda}$, we only
observe a subset of $K_{\lambda,N_{\lambda}}$ , $1\leq K_{\lambda,N_{\lambda}}\leq N_{\lambda}$,
votes. We can assume without loss of generality that these votes are
the first $K_{\lambda,N_{\lambda}}$ votes belonging to group $\lambda$
since the random variables $X_{\lambda1},\ldots,X_{\lambda N_{\lambda}}$
representing the votes from group $\lambda$ are exchangeable (see
Definition \ref{def:exchange} and Lemma \ref{lem:exchange})\footnote{In fact, a generalisation is possible and frequently useful in practice:
we can assume that the identities of the $K_{\lambda,N_{\lambda}}$
voters from group $\lambda$ vary between observations. The only essential
assumption is that each observation $t\in\IN_{n}$ in the sample has
the same number $K_{\lambda,N_{\lambda}}$ of votes from each group
$\lambda$.}. Set $\boldsymbol{K}_{\boldsymbol{N}}\coloneq\left(K_{1,N_{1}},\ldots,K_{M,N_{M}}\right)$
and $\bar{K}_{\boldsymbol{N}}\coloneq\sum_{\lambda=1}^{M}K_{\lambda,N_{\lambda}}$.

Each sample takes values in the space
\[
\Omega_{\bar{K}_{\boldsymbol{N}}}^{n}\coloneq\prod_{i=1}^{n}\Omega_{\bar{K}_{\boldsymbol{N}}}.
\]

\begin{defn}
\label{def:alpha}Let $\left(K_{\lambda,N_{\lambda}}\right)_{N_{\lambda}\in\IN}$,
$\lambda\in\IN_{M}$, be sequences with the properties $K_{\lambda,N_{\lambda}}\in\IN$
and $1\leq K_{\lambda,N_{\lambda}}\leq N_{\lambda}$, $N_{\lambda}\in\IN$.
We define
\[
\alpha_{\lambda}\coloneq\lim_{N_{\lambda}\rightarrow\infty}\frac{K_{\lambda,N_{\lambda}}}{N_{\lambda}},\quad\lambda\in\IN_{M},
\]
if each limit exists.
\end{defn}

We also define the sum of all observed votes:
\begin{defn}
\label{def:sub_sum}Let $\left(K_{\lambda,N_{\lambda}}\right)_{N_{\lambda}\in\IN}$,
$\lambda\in\IN_{M}$, be sequences with the properties $K_{\lambda,N_{\lambda}}\in\IN$
and $1\leq K_{\lambda,N_{\lambda}}\leq N_{\lambda}$, $N_{\lambda}\in\IN$.
We define for all $\lambda\in\IN_{M}$,
\[
\Sigma_{\lambda,N_{\lambda}}\coloneq\sum_{i=1}^{K_{\lambda,N_{\lambda}}}X_{\lambda i}
\]
and
\[
\boldsymbol{\Sigma}_{\boldsymbol{N}}\coloneq\left(\Sigma_{1,N_{1}},\ldots,\Sigma_{M,N_{M}}\right).
\]
\end{defn}

\begin{rem}
We note that the distribution of the random variable $\Sigma_{\lambda,N_{\lambda}}$
depends on both $K_{\lambda,N_{\lambda}}$ and $N_{\lambda}$. For
readability's sake, we omit the subindex $K_{\lambda,N_{\lambda}}$.
\end{rem}

We will next specify intervals for the estimation of the parameters
$\beta_{\lambda}$.
\begin{defn}
\label{def:intervals}Given constants $0\leq b_{1}<1<b_{2}$ such
that 
\begin{equation}
\frac{b_{1}}{1-b_{1}}\frac{1}{N}+\boldsymbol{C}_{\textup{high}}\left(\frac{\ln N}{N}\right)^{2}<m\left(b_{2}\right)^{2}-\boldsymbol{C}_{\textup{low}}\frac{\left(\ln N\right)^{\frac{3}{2}}}{\sqrt{N}},\label{eq:separation}
\end{equation}
define the intervals
\begin{align*}
I_{h} & \coloneq\left[0,b_{1}\right],\quad I{}_{c}\coloneq\left(b_{1},b_{2}\right),\quad I{}_{l}\coloneq\left[b_{2},\infty\right),\\
J_{h} & \coloneq\left[-1,\frac{b_{1}}{1-b_{1}}\frac{1}{N}+\boldsymbol{C}_{\textup{high}}\left(\frac{\ln N}{N}\right)^{2}\right],\\
J_{c} & \coloneq\left(\frac{b_{1}}{1-b_{1}}\frac{1}{N}+\boldsymbol{C}_{\textup{high}}\left(\frac{\ln N}{N}\right)^{2},m\left(b_{2}\right)^{2}-\boldsymbol{C}_{\textup{low}}\frac{\left(\ln N\right)^{\frac{3}{2}}}{\sqrt{N}}\right),\\
J_{l} & \coloneq\left[m\left(b_{2}\right)^{2}-\boldsymbol{C}_{\textup{low}}\frac{\left(\ln N\right)^{\frac{3}{2}}}{\sqrt{N}},1\right].
\end{align*}
\end{defn}

\begin{rem}
\label{rem:intervals}Condition (\ref{eq:separation}) holds for all
$N$ large enough since $m\left(\beta\right)>0$ for all $\beta>1$
by Lemma \ref{lem:m_beta_increasing}. By Proposition \ref{prop:appr_correlations},
we have $\IE_{\beta,N}X_{1}X_{2}\in J_{k}$ for all $\beta\in I_{k}$,
$k\in\left\{ h,l\right\} $. In fact, Proposition \ref{prop:appr_correlations}
states the stronger condition that for all $\beta\in I_{k}$, $\IE_{\beta,N}X_{1}X_{2}$
lies in the interior of $J_{k}$ for $k\in\left\{ h,l\right\} $.
We will define $\hat{\boldsymbol{\gamma}}_{\boldsymbol{N}}$ (Definitions
\ref{def:pair_2} and \ref{def:pair_K}) supposing true parameter
values $\beta_{\lambda}$ in $I_{h}\cup I_{l}$ for each group, where
we will assume that for each $\lambda$ condition (\ref{eq:separation})
holds for $N_{\lambda}$ instead of the generic $N$.
\end{rem}

\begin{defn}
\label{def:intervals-1}Let $\left(K_{N}\right)_{N\in\IN}$ be a divergent
sequence with the properties $K_{N}\in\IN$ and $1\leq K_{N}\leq N$,
$N\in\IN$, and let $\alpha\in\left[0,1\right]$ be as in Definition
\ref{def:alpha}. Given constants $0\leq b_{1}<1<b_{2}$ such that
\begin{equation}
\frac{1-\left(1-\alpha\right)b_{1}}{1-b_{1}}K_{N}+\boldsymbol{D}_{\textup{high}}\sqrt{K_{N}}<m\left(b_{2}\right)^{2}K_{N}^{2}-\boldsymbol{D}_{\textup{low}}\frac{\left(\ln N\right)^{\frac{3}{2}}}{\sqrt{N}}K_{N}^{2},\label{eq:separation-1}
\end{equation}
define the intervals
\begin{align*}
I'_{h} & \coloneq\left[0,b_{1}\right],\quad I'_{c}\coloneq\left(b_{1},b_{2}\right),\quad I'_{l}\coloneq\left[b_{2},\infty\right),\\
J'_{h} & \coloneq\left[\min\textup{Range}\left(\Sigma_{K_{N}}^{2}\right),\frac{1-\left(1-\alpha\right)b_{1}}{1-b_{1}}K_{N}+\boldsymbol{D}_{\textup{high}}\sqrt{K_{N}}\right],\\
J'_{c} & \coloneq\left(\frac{1-\left(1-\alpha\right)b_{1}}{1-b_{1}}K_{N}+\boldsymbol{D}_{\textup{high}}\sqrt{K_{N}},m\left(b_{2}\right)^{2}K_{N}^{2}-\boldsymbol{D}_{\textup{low}}\frac{\left(\ln N\right)^{\frac{3}{2}}}{\sqrt{N}}K_{N}^{2}\right),\\
J'_{l} & \coloneq\left[m\left(b_{2}\right)^{2}K_{N}^{2}-\boldsymbol{D}_{\textup{low}}\frac{\left(\ln N\right)^{\frac{3}{2}}}{\sqrt{N}}K_{N}^{2},\infty\right).
\end{align*}
\end{defn}

\begin{rem}
\label{rem:intervals-1}Similarly to Remark \ref{rem:intervals},
Proposition \ref{prop:expec_Sigma} states that for all $\beta\in I'_{k}$,
$\IE_{\beta,N}S^{2}$ lies in the interior of $J'_{k}$ for $k\in\left\{ h,l\right\} $.
We will define $\hat{\boldsymbol{\zeta}}_{\boldsymbol{N}}$ (Definition
\ref{def:zeta^hat}) supposing true parameter values $\beta_{\lambda}$
in $I'_{h}\cup I'_{l}$ for each group, where we will assume that
for each $\lambda$ condition (\ref{eq:separation-1}) holds for $N_{\lambda}$
and $K_{\lambda,N_{\lambda}}$instead of the generic $N$ and $K_{N}$.
\end{rem}

\subsection{Pair Correlation Estimators}

The first family of estimators we will define is based on the observation
of how pairs of voters interact with each other, i.e. whether they
tend to agree in their votes, and if so, how strongly.

See Proposition \ref{prop:appr_correlations} concerning the asymptotic
behaviour of $\IE_{\beta_{\lambda},N_{\lambda}}X_{\lambda1}X_{\lambda2}$.
We will use the limits presented therein to define the constants $\tilde{\gamma}_{N}$,
$N\in\IN$. Recall also Definition \ref{def:intervals} of the intervals
$I_{h}$ and $I_{l}$.
\begin{defn}
\label{def:gamma_tilde}Let for all $N\in\IN$ the constant $\tilde{\gamma}_{N}$
be defined by
\[
\IE_{\beta,N}X_{1}X_{2}=\begin{cases}
\frac{\tilde{\gamma}_{N}}{1-\tilde{\gamma}_{N}}\frac{1}{N} & \textup{if }\beta\in I{}_{h},\\
m\left(\tilde{\gamma}_{N}\right)^{2} & \textup{if }\beta\in I{}_{l}.
\end{cases}
\]
Set the vector $\tilde{\gamma}_{\boldsymbol{N}}$ equal to
\[
\left(\tilde{\gamma}_{N_{1}},\ldots,\tilde{\gamma}_{N_{M}}\right).
\]
\end{defn}

We define the estimator $P_{K_{N}}$ based on all $\left(\begin{array}{c}
K_{N}\\
2
\end{array}\right)$ pair correlations among the $K_{N}$ votes as follows: we first define
a statistic $P_{K_{N}}$ and then use it to define the estimator $\hat{\gamma}_{K_{N}}$
for $\beta$.
\begin{defn}
\label{def:PKN}Let $n,N\in\IN$ and $\left(x^{(1)},\ldots,x^{(n)}\right)\in\Omega_{K_{N}}^{n}$.
Then we define
\begin{align*}
\boldsymbol{P}_{\boldsymbol{K}_{\boldsymbol{N}}}\left(x^{(1)},\ldots,x^{(n)}\right) & \coloneq\frac{1}{n}\sum_{t=1}^{n}\left(\frac{1}{K_{1,N_{1}}\left(K_{1,N_{1}}-1\right)}\sum_{1\leq i,j\leq K_{1,N_{1}},i\neq j}x_{1i}^{(t)}x_{1j}^{(t)},\right..\\
 & \quad\left.\ldots,\frac{1}{K_{M,N_{M}}\left(K_{M,N_{M}}-1\right)}\sum_{1\leq i,j\leq K_{M,N_{M}},i\neq j}x_{Mi}^{(t)}x_{Mj}^{(t)}\right)
\end{align*}
\end{defn}

\begin{notation}\label{notation:infty}We will write $\left[-\infty,\infty\right]$
for the compactification $\IR\cup\left\{ -\infty,\infty\right\} $
and $\left[0,\infty\right]$ for $\left[0,\infty\right)\cup\left\{ \infty\right\} $.\end{notation}
\begin{defn}
\label{def:pair_K}Let $\left(K_{\lambda,N_{\lambda}}\right)_{N_{\lambda}\in\IN}$,
$\lambda\in\IN_{M}$, be divergent sequences with the properties $K_{\lambda,N_{\lambda}}\in\IN$
and $1\leq K_{\lambda,N_{\lambda}}\leq N_{\lambda}$, $N_{\lambda}\in\IN$.
Let $\boldsymbol{\beta}\geq0$, $n\in\IN$, $\boldsymbol{N}\in\IN^{M}$,
and $\left(x^{(1)},\ldots,x^{(n)}\right)\in\Omega_{\bar{K}_{\boldsymbol{N}}}^{n}$.
Let $0\leq b_{1}<1<b_{2}$ be constants satisfying the condition (\ref{eq:separation}),
and consider the intervals from Definition \ref{def:intervals}. Then
we define the \emph{estimator} $\hat{\boldsymbol{\gamma}}_{\boldsymbol{K}_{\boldsymbol{N}}}:\Omega_{\bar{K}_{\boldsymbol{N}}}^{n}\rightarrow\left[-\infty,\infty\right]$
\emph{based on the correlations between the first $K_{N}$ votes},
for all $\left(x^{(1)},\ldots,x^{(n)}\right)\in\Omega_{K_{N}}^{n}$,
by
\begin{enumerate}
\item If $\left(\boldsymbol{P}_{\boldsymbol{K}_{\boldsymbol{N}}}\left(x^{(1)},\ldots,x^{(n)}\right)\right)_{\lambda}\in J{}_{h}$,
then
\[
\left(\hat{\boldsymbol{\gamma}}_{\boldsymbol{K}_{\boldsymbol{N}}}\left(x^{(1)},\ldots,x^{(n)}\right)\right)_{\lambda}\coloneq\begin{cases}
\frac{N_{\lambda}P_{K_{N}}\left(x^{(1)},\ldots,x^{(n)}\right)}{N_{\lambda}P_{K_{N}}\left(x^{(1)},\ldots,x^{(n)}\right)+1} & \textup{if }\left(\boldsymbol{P}_{\boldsymbol{K}_{\boldsymbol{N}}}\left(x^{(1)},\ldots,x^{(n)}\right)\right)_{\lambda}>-\frac{1}{N_{\lambda}},\\
-\infty & \textup{if }\left(\boldsymbol{P}_{\boldsymbol{K}_{\boldsymbol{N}}}\left(x^{(1)},\ldots,x^{(n)}\right)\right)_{\lambda}\leq-\frac{1}{N_{\lambda}}.
\end{cases}
\]
\item If $\left(\boldsymbol{P}_{\boldsymbol{K}_{\boldsymbol{N}}}\left(x^{(1)},\ldots,x^{(n)}\right)\right)_{\lambda}\in J{}_{l}$,
then $\hat{\gamma}_{K_{N}}>1$ is given by the unique value for which
$\left(\boldsymbol{P}_{\boldsymbol{K}_{\boldsymbol{N}}}\left(x^{(1)},\ldots,x^{(n)}\right)\right)_{\lambda}=m\left(\left(\hat{\boldsymbol{\gamma}}_{\boldsymbol{K}_{\boldsymbol{N}}}\left(x^{(1)},\ldots,x^{(n)}\right)\right)_{\lambda}\right)^{2}$
is satisfied.
\item If $\left(\boldsymbol{P}_{\boldsymbol{K}_{\boldsymbol{N}}}\left(x^{(1)},\ldots,x^{(n)}\right)\right)_{\lambda}\in J{}_{c}$,
then we say there is insufficient evidence in the sample to conclude
that $\beta_{\lambda}$ is significantly different from 1, and $\left(\hat{\boldsymbol{\gamma}}_{\boldsymbol{K}_{\boldsymbol{N}}}\left(x^{(1)},\ldots,x^{(n)}\right)\right)_{\lambda}\coloneq\textup{u}$
is undefined.
\end{enumerate}
\end{defn}

The next theorem generalises Theorem \ref{thm:properties_gamma_2}
and gives the properties of the estimator $\hat{\boldsymbol{\gamma}}_{\boldsymbol{K}_{\boldsymbol{N}}}$.
Recall Definitions \ref{def:alpha} and \ref{def:sub_sum}.
\begin{thm}
\label{thm:properties_gamma_K}The following statements hold:
\begin{enumerate}
\item For fixed $N\in\IN$, $\hat{\boldsymbol{\gamma}}_{\boldsymbol{K}_{\boldsymbol{N}}}\xrightarrow[n\rightarrow\infty]{\textup{p}}\tilde{\gamma}_{\boldsymbol{N}}$.
\item $\sqrt{n}\left(\hat{\boldsymbol{\gamma}}_{\boldsymbol{K}_{\boldsymbol{N}}}-\tilde{\gamma}_{\boldsymbol{N}}\right)\xrightarrow[n\rightarrow\infty]{\textup{d}}\mathcal{N}\left(0,\varDelta_{\boldsymbol{N}}\right)$,
where the covariance matrix $\varDelta_{\boldsymbol{N}}$ is diagonal.
For all $\lambda\in\IN_{M}$, we have:
\begin{enumerate}
\item If $\beta_{\lambda}\in I'_{h}$, then
\[
\left(\varDelta_{\boldsymbol{N}}\right)_{\lambda\lambda}=\left(1-\tilde{\gamma}_{N_{\lambda}}\right)^{4}\left(\frac{N_{\lambda}}{K_{\lambda,N_{\lambda}}-1}\right)^{2}\IV\frac{\Sigma_{\lambda,N_{\lambda}}^{2}}{K_{\lambda,N_{\lambda}}}
\]
and, if $\alpha_{\lambda}>0$, then
\begin{align*}
\lim_{N_{\lambda}\rightarrow\infty}\left(1-\tilde{\gamma}_{N_{\lambda}}\right)^{4}\left(\frac{N_{\lambda}}{K_{\lambda,N_{\lambda}}-1}\right)^{2}\IV_{\beta,N}\frac{\Sigma_{\lambda,N_{\lambda}}^{2}}{K_{\lambda,N_{\lambda}}} & =\frac{2\left(1-\beta_{\lambda}\right)^{2}\left(1-\left(1-\alpha_{\lambda}\right)\beta_{\lambda}\right)^{2}}{\alpha_{\lambda}^{2}}.
\end{align*}
\item If $\beta_{\lambda}\in I'_{l}$, then
\[
\left(\varDelta_{\boldsymbol{N}}\right)_{\lambda\lambda}=\left(\frac{K_{\lambda,N_{\lambda}}}{K_{\lambda,N_{\lambda}}-1}\right)^{2}\IV_{\beta,N}\left(\frac{\Sigma_{\lambda,N_{\lambda}}}{K_{\lambda,N_{\lambda}}}\right)^{2}\xrightarrow[N_{\lambda}\rightarrow\infty]{}0.
\]
\end{enumerate}
\end{enumerate}
\end{thm}

\subsection{Estimators Based on the Maximum Likelihood Optimality Condition}

In this subsection, we apply the optimality condition (\ref{eq:opt})
of the maximum likelihood estimator for $\boldsymbol{\beta}$ to the
situation where we only have access to the first $K_{\lambda,N_{\lambda}},1\leq K_{\lambda,N_{\lambda}}\leq N,$
votes from each group $\lambda$ on each observation. Then, instead
of (\ref{eq:opt}), we have the condition
\[
\IE_{\hat{\zeta_{\lambda}},N}\Sigma_{\lambda,N_{\lambda}}^{2}=\frac{1}{n}\sum_{t=1}^{n}\left(\sum_{i=1}^{K_{\lambda,N_{\lambda}}}x_{\lambda i}^{(t)}\right)^{2}
\]
for each $\lambda\in\IN_{M}$ for any sample $\left(x^{(1)},\ldots,x^{(n)}\right)\in\Omega_{\bar{K}_{\boldsymbol{N}}}^{n}$
yielding an estimator $\hat{\boldsymbol{\zeta}}$ for $\boldsymbol{\beta}$
we will define formally in Definition \ref{def:zeta^hat}. We need
an asymptotic approximation for the moment $\IE_{\hat{\zeta},N}\Sigma_{\lambda,N_{\lambda}}^{2}$,
which is supplied by Proposition \ref{prop:expec_Sigma}.
\begin{defn}
\label{def:TKN}We define the statistic $\boldsymbol{T}_{\boldsymbol{K}_{\boldsymbol{N}}}:\Omega_{\bar{K}_{\boldsymbol{N}}}^{n}\rightarrow\IR$
for any realisation of the sample $\left(x^{(1)},\ldots,x^{(n)}\right)\in\Omega_{\bar{K}_{\boldsymbol{N}}}^{n}$
by
\[
\boldsymbol{T}_{\boldsymbol{K}_{\boldsymbol{N}}}\left(x^{(1)},\ldots,x^{(n)}\right)\coloneq\frac{1}{n}\sum_{t=1}^{n}\left(\left(\sum_{i=1}^{K_{1,N_{1}}}x_{1i}^{(t)}\right)^{2},\ldots,\left(\sum_{i=1}^{K_{M,N_{M}}}x_{Mi}^{(t)}\right)^{2}\right).
\]
\end{defn}

\begin{defn}
\label{def:zeta^hat}Let $\boldsymbol{\beta}\geq0$, let $0\leq b_{1}<1<b_{2}$
be constants satisfying the condition (\ref{eq:separation-1}), and
consider the intervals from Definition \ref{def:intervals-1}. The
estimator $\hat{\boldsymbol{\zeta}}_{\boldsymbol{K}_{\boldsymbol{N}}}:\Omega_{\bar{K}_{\boldsymbol{N}}}^{n}\rightarrow\left[-\infty,\infty\right]$
is defined by
\begin{enumerate}
\item If $\left(\boldsymbol{T}_{\boldsymbol{K}_{\boldsymbol{N}}}\left(x^{(1)},\ldots,x^{(n)}\right)\right)_{\lambda}\in J'_{h}$,
then
\[
\left(\hat{\boldsymbol{\zeta}}_{\boldsymbol{K}_{\boldsymbol{N}}}\left(x^{(1)},\ldots,x^{(n)}\right)\right)_{\lambda}\coloneq\begin{cases}
\frac{K_{\lambda,N_{\lambda}}-\left(\boldsymbol{T}_{\boldsymbol{K}_{\boldsymbol{N}}}\left(x^{(1)},\ldots,x^{(n)}\right)\right)_{\lambda}}{K_{\lambda,N_{\lambda}}\left(1-\alpha_{\lambda}\right)-\left(\boldsymbol{T}_{\boldsymbol{K}_{\boldsymbol{N}}}\left(x^{(1)},\ldots,x^{(n)}\right)\right)_{\lambda}} & \textup{if }\left(\boldsymbol{T}_{\boldsymbol{K}_{\boldsymbol{N}}}\left(x^{(1)},\ldots,x^{(n)}\right)\right)_{\lambda}>-K_{\lambda,N_{\lambda}}\left(1-\alpha_{\lambda}\right),\\
-\infty & \textup{if }\left(\boldsymbol{T}_{\boldsymbol{K}_{\boldsymbol{N}}}\left(x^{(1)},\ldots,x^{(n)}\right)\right)_{\lambda}\leq-K_{\lambda,N_{\lambda}}\left(1-\alpha_{\lambda}\right).
\end{cases}
\]
\item If $\left(\boldsymbol{T}_{\boldsymbol{K}_{\boldsymbol{N}}}\left(x^{(1)},\ldots,x^{(n)}\right)\right)_{\lambda}\in J'_{l}$,
then $\left(\hat{\boldsymbol{\zeta}}_{\boldsymbol{K}_{\boldsymbol{N}}}\left(x^{(1)},\ldots,x^{(n)}\right)\right)_{\lambda}>1$
is given by the unique value for which $\left(\boldsymbol{T}_{\boldsymbol{K}_{\boldsymbol{N}}}\left(x^{(1)},\ldots,x^{(n)}\right)\right)_{\lambda}=m\left(\left(\hat{\boldsymbol{\zeta}}_{\boldsymbol{K}_{\boldsymbol{N}}}\left(x^{(1)},\ldots,x^{(n)}\right)\right)_{\lambda}\right)^{2}K_{N}^{2}$
is satisfied.
\item If $\left(\boldsymbol{T}_{\boldsymbol{K}_{\boldsymbol{N}}}\left(x^{(1)},\ldots,x^{(n)}\right)\right)_{\lambda}\in J'_{c}$,
then we say there is insufficient evidence in the sample to conclude
that $\beta_{\lambda}$ is significantly different from 1, and $\left(\hat{\boldsymbol{\zeta}}_{\boldsymbol{K}_{\boldsymbol{N}}}\left(x^{(1)},\ldots,x^{(n)}\right)\right)_{\lambda}\coloneq\textup{u}$
is undefined.
\end{enumerate}
\end{defn}

\begin{rem}
If $\alpha_{\lambda}=0$, then in the high temperature regime $\beta_{\lambda}<1$,
we do not obtain an estimator for $\beta_{\lambda}$ as the value
of $\left(\hat{\boldsymbol{\zeta}}_{\boldsymbol{K}_{\boldsymbol{N}}}\right)_{\lambda}$
will be equal to 1 for any $\left(x^{(1)},\ldots,x^{(n)}\right)\in\Omega_{K_{N}}^{n}$
with $\left(\boldsymbol{T}_{\boldsymbol{K}_{\boldsymbol{N}}}\left(x^{(1)},\ldots,x^{(n)}\right)\right)_{\lambda}\in J'_{h}$.
The reason for this is that for $\alpha_{\lambda}=0$, the sequence
of random variables $\left(\frac{\Sigma_{K_{N}}^{2}}{K_{N}}\right)_{N\in\IN}$
converges in distribution to $\mathcal{N}\left(0,1\right)$, and thus
the asymptotic approximation for $\IE_{\beta_{\lambda},N_{\lambda}}\Sigma_{\lambda,N_{\lambda}}^{2}$
in Proposition \ref{prop:expec_Sigma} contains no information regarding
$\beta_{\lambda}$. The same problem does not arise in the low temperature
regime $\beta_{\lambda}>1$, where even if $\alpha_{\lambda}=0$,
we can estimate $\beta_{\lambda}$ using the estimator $\hat{\zeta}_{K_{N}}$.
If $\alpha_{\lambda}=1$, we recover the estimator from Definition
10 in \cite{BalNauTo2025} based on the entire population in both
regimes, in the sense that the estimator $\hat{\boldsymbol{\zeta}}_{\boldsymbol{K}_{\boldsymbol{N}}}$
has the same functional and is asymptotically equal to said estimator.
This means that discarding an asymptotically insignificant number
of votes from the observations does not affect the estimation of $\beta_{\lambda}$
via the maximum likelihood principle.
\end{rem}

\begin{defn}
\label{def:zeta_tilde}Let the intervals $I'_{h}$ and $I'_{l}$ be
as in Definition \ref{def:intervals-1}. Let, for all $N\in\IN$ and
all $\beta\in I'_{h}\cup I'_{l}$, $\tilde{\zeta}_{N}\geq0$ be the
value which satisfies
\[
\IE_{\beta,N}\Sigma_{1,N}^{2}=\begin{cases}
\frac{1-\left(1-\alpha\right)\tilde{\zeta}_{N}}{1-\tilde{\zeta}_{N}}K_{1,N} & \textup{if }\beta\in I'_{h},\\
m\left(\tilde{\zeta}_{N}\right)^{2}K_{1,N}^{2} & \textup{if }\beta\in I'_{l}.
\end{cases}
\]
Set the vector $\tilde{\boldsymbol{\zeta}}_{\boldsymbol{N}}$ equal
to
\[
\left(\tilde{\zeta}_{N_{1}},\ldots,\tilde{\zeta}_{N_{M}}\right).
\]
\end{defn}

\begin{thm}
\label{thm:properties_zeta_K}Suppose $\alpha>0$. Then the following
statements hold:
\begin{enumerate}
\item For fixed $N\in\IN$, $\hat{\boldsymbol{\zeta}}_{\boldsymbol{K}_{\boldsymbol{N}}}\xrightarrow[n\rightarrow\infty]{\textup{p}}\tilde{\boldsymbol{\zeta}}_{\boldsymbol{N}}$.
\item $\sqrt{n}\left(\hat{\boldsymbol{\zeta}}_{\boldsymbol{K}_{\boldsymbol{N}}}-\tilde{\boldsymbol{\zeta}}_{\boldsymbol{N}}\right)\xrightarrow[n\rightarrow\infty]{\textup{d}}\mathcal{N}\left(0,\Psi_{\boldsymbol{N}}\right)$,
where the covariance matrix $\Psi_{\boldsymbol{N}}$ is diagonal.
The following statements hold:
\begin{enumerate}
\item Let $\beta_{\lambda}\in I'_{h}$. Then
\begin{align*}
\left(\Psi_{\boldsymbol{N}}\right)_{\lambda\lambda} & =\left(1-\tilde{\zeta}_{N}\right)^{4}\left(\frac{N}{K_{\lambda,N_{\lambda}}-1}\right)^{2}\IV\frac{\Sigma_{\lambda,N_{\lambda}}^{2}}{K_{\lambda,N_{\lambda}}}\xrightarrow[N_{\lambda}\rightarrow\infty]{}\frac{2\left(1-\beta_{\lambda}\right)^{2}\left(1-\left(1-\alpha_{\lambda}\right)\beta_{\lambda}\right)^{2}}{\alpha_{\lambda}^{2}}.
\end{align*}
\item Let $\beta_{\lambda}\in I'_{l}$. Then
\begin{align*}
\left(\Psi_{\boldsymbol{N}}\right)_{\lambda\lambda}=\frac{1}{\left(2m\left(\beta_{\lambda}\right)m'\left(\beta_{\lambda}\right)\right)^{2}}\IV\left(\frac{\Sigma_{\lambda,N_{\lambda}}}{K_{\lambda,N_{\lambda}}}\right)^{2} & \xrightarrow[N_{\lambda}\rightarrow\infty]{}0.
\end{align*}
\end{enumerate}
\end{enumerate}
\end{thm}

\subsection{Large Population Asymptotic Equivalence of the Estimators $\hat{\boldsymbol{\gamma}}_{\boldsymbol{K}_{\boldsymbol{N}}}$
and $\hat{\boldsymbol{\zeta}}_{\boldsymbol{K}_{\boldsymbol{N}}}$}

We show how the estimators based on a subset of votes that use the
empirical pair correlations and those that use the maximum likelihood
condition are asymptotically equivalent as the group populations $N_{\lambda}$
become very large.

In preparation for Theorem \ref{thm:asymp_equiv}, we define some
sets.
\begin{defn}
\label{def:sets}Assume $\alpha\in\left(0,1\right)$. Let $b>0$ and
define the following sets:
\begin{align*}
A_{N_{\lambda},n}\left(\lambda\right) & \coloneq\left\{ x\in\Omega_{\bar{K}_{\boldsymbol{N}}}^{n}\,|\,\left(\boldsymbol{P}_{\boldsymbol{K}_{\boldsymbol{N}}}\left(x\right)\right)_{\lambda}\in\left(-\infty,-\frac{1}{N_{\lambda}}\right]\right\} ,\\
A'_{N_{\lambda},n}\left(\lambda\right) & \coloneq\left\{ x\in\Omega_{\bar{K}_{\boldsymbol{N}}}^{n}\,|\,\left(\boldsymbol{T}_{\boldsymbol{K}_{\boldsymbol{N}}}\left(x\right)\right)_{\lambda}\in\left(-\infty,-K_{\lambda,N_{\lambda}}\left(1-\alpha_{\lambda}\right)\right]\right\} ,\\
B_{N_{\lambda},n}\left(\lambda\right) & \coloneq\left\{ x\in\Omega_{\bar{K}_{\boldsymbol{N}}}^{n}\,|\,\left(\boldsymbol{P}_{\boldsymbol{K}_{\boldsymbol{N}}}\left(x\right)\right)_{\lambda}\in\left(-\frac{b}{1+b}\frac{1}{N_{\lambda}},\infty\right)\cap J_{h}\right\} ,\\
B'_{N_{\lambda},n}\left(\lambda\right) & \coloneq\left\{ x\in\Omega_{\bar{K}_{\boldsymbol{N}}}^{n}\,|\,\left(\boldsymbol{T}_{\boldsymbol{K}_{\boldsymbol{N}}}\left(x\right)\right)_{\lambda}\in\left(-K_{\lambda,N_{\lambda}}\frac{1+\left(1-\alpha_{\lambda}\right)b}{1+b},\infty\right)\cap J'_{h}\right\} ,
\end{align*}
and
\[
H_{N_{\lambda},n}\left(\lambda\right)\coloneq B_{N_{\lambda},n}\left(\lambda\right)\cap B'_{N,n}\left(\lambda\right).
\]
Next, let
\begin{align*}
D_{N_{\lambda},n}\left(\lambda\right) & \coloneq\left\{ x\in\Omega_{\bar{K}_{\boldsymbol{N}}}^{n}\,|\,\left(\boldsymbol{P}_{\boldsymbol{K}_{\boldsymbol{N}}}\left(x\right)\right)_{\lambda}\in\left[m\left(b_{2}\right)^{2},m\left(b\right)^{2}\right]\right\} ,\\
D'_{N_{\lambda},n}\left(\lambda\right) & \coloneq\left\{ x\in\Omega_{\bar{K}_{\boldsymbol{N}}}^{n}\,|\,\left(\boldsymbol{T}_{\boldsymbol{K}_{\boldsymbol{N}}}\left(x\right)\right)_{\lambda}\in K_{\lambda,N_{\lambda}}^{2}\left[m\left(b_{2}\right)^{2},m\left(b\right)^{2}\right]\right\} ,
\end{align*}
and
\[
L_{N_{\lambda},n}\left(\lambda\right)\coloneq D_{N_{\lambda},n}\left(\lambda\right)\cap D'_{N_{\lambda},n}\left(\lambda\right).
\]
\end{defn}

\begin{thm}
\label{thm:asymp_equiv}Let $\left(K_{\lambda,N_{\lambda}}\right)_{N_{\lambda}\in\IN}$,
$\lambda\in\IN_{M}$, be divergent sequences with the properties $K_{\lambda,N_{\lambda}}\in\IN$
and $1\leq K_{\lambda,N_{\lambda}}\leq N_{\lambda}$, $N_{\lambda}\in\IN$.
Let $\lambda\in\IN_{M}$ and assume $\alpha_{\lambda}\in\left(0,1\right)$,
$\beta_{\lambda}\in I_{h}$ . Fix $b>0$, and consider the sets from
Definition \ref{def:sets}. Then
\[
\hat{\gamma}_{K_{N}}\left(x\right)=\hat{\zeta}_{K_{N}}\left(x\right)=-\infty,\quad x\in A_{N_{\lambda},n}\left(\lambda\right)\cap A'_{N_{\lambda},n}\left(\lambda\right),
\]
and
\[
\sup_{n\in\IN}\sup_{x\in H_{N_{\lambda},n}\left(\lambda\right)}\left|\left(\hat{\boldsymbol{\gamma}}_{\boldsymbol{K}_{\boldsymbol{N}}}\left(x\right)\right)_{\lambda}-\left(\hat{\boldsymbol{\zeta}}_{\boldsymbol{K}_{\boldsymbol{N}}}\left(x\right)\right)_{\lambda}\right|\leq\frac{1}{\alpha_{\lambda}}\left(1+b\right)^{2}\frac{b_{1}}{1-b_{1}}\left|\alpha_{\lambda}-\frac{K_{\lambda,N_{\lambda}}}{N_{\lambda}}+\frac{1}{N_{\lambda}}\right|.
\]

Let $\beta_{\lambda}\in I_{l}$, fix $b>\beta_{\lambda}$, and consider
the sets from Definition \ref{def:sets}. Then
\[
\sup_{n\in\IN}\sup_{x\in L_{N_{\lambda},n}\left(\lambda\right)}\left|\left(\hat{\boldsymbol{\gamma}}_{\boldsymbol{K}_{\boldsymbol{N}}}\left(x\right)\right)_{\lambda}-\hat{\boldsymbol{\zeta}}_{\boldsymbol{K}_{\boldsymbol{N}}}\left(x\right)\right|\leq C\frac{1}{K_{\lambda,N_{\lambda}}}.
\]
\end{thm}

\begin{rem}
\begin{enumerate}
\item The reason we need to restrict the range over which we supply a uniform
upper bound for both the high temperature and the low temperature
regime is technical: the derivative of the function $t\mapsto\frac{t}{t+1}$
is unbounded above on the infinite interval $\left(-\infty,0\right)$
and so is the derivative of the function $m^{-1}$ on $\left(0,\infty\right)$.
\item As a consequence of the above, no uniform upper bound for the Euclidean
distance between the estimators $\hat{\boldsymbol{\gamma}}_{\boldsymbol{K}_{\boldsymbol{N}}}$
and $\hat{\boldsymbol{\zeta}}_{\boldsymbol{K}_{\boldsymbol{N}}}$
over all possible samples exists. We note that, as a corollary to
the respective first statements of Theorems \ref{thm:properties_gamma_K}
and \ref{thm:properties_zeta_K},
\[
\left\Vert \hat{\boldsymbol{\gamma}}_{\boldsymbol{K}_{\boldsymbol{N}}}-\hat{\boldsymbol{\zeta}}_{\boldsymbol{K}_{\boldsymbol{N}}}\right\Vert \xrightarrow[n\rightarrow\infty]{\textup{p}}0
\]
holds.
\item By Definition \ref{def:alpha}, $\left|\alpha_{\lambda}-\frac{K_{\lambda,N_{\lambda}}}{N_{\lambda}}\right|\xrightarrow[N\rightarrow\infty]{}0$
is satisfied. The convergence speed of $\left|\alpha_{\lambda}-\frac{K_{\lambda,N_{\lambda}}}{N_{\lambda}}\right|$
matters for the upper bound for $\left\Vert \hat{\boldsymbol{\gamma}}_{\boldsymbol{K}_{\boldsymbol{N}}}\left(x\right)-\hat{\boldsymbol{\zeta}}_{\boldsymbol{K}_{\boldsymbol{N}}}\left(x\right)\right\Vert $
which is of order $O\left(\left|\alpha_{\lambda}-\frac{K_{\lambda,N_{\lambda}}}{N_{\lambda}}\right|+\frac{1}{N_{\lambda}}\right)$.
In the low temperature regime, the limit $\alpha_{\lambda}$ of the
proportion of votes we have access to does not matter, and nor does
the convergence speed of \textbf{$\left|\alpha_{\lambda}-\frac{K_{\lambda,N_{\lambda}}}{N_{\lambda}}\right|$}
to 0.
\end{enumerate}
\end{rem}

\section{\label{sec:Proofs}Proofs of Theorems \ref{thm:properties_gamma_K},
\ref{thm:properties_zeta_K}, and \ref{thm:asymp_equiv}}

Due to the non-interaction assumption in Definition \ref{def:CWM},
we can treat the estimation of the coupling constant belonging to
a single group separately. To simplify the somewhat involved notation
of the variables we use in the article, we omit the subindex identifying
the group in question. Thus we will write $T_{K_{N}}$ instead of
$\left(\boldsymbol{T}_{\boldsymbol{K}_{\boldsymbol{N}}}\left(x^{(1)},\ldots,x^{(n)}\right)\right)_{\lambda}$,
$\Sigma_{N}$ for $\Sigma_{\lambda,N_{\lambda}}$, etc. 

\subsection{Asymptotic Behaviour of the Partial Sums $\Sigma_{K_{N}}^{2k}$}
\begin{prop}
\label{prop:expec_Sigma}For all $\beta\in\IR,\beta\neq1$, and all
$k\in\IN$, the moment $\IE_{\beta,N}\Sigma_{\lambda,N_{\lambda}}^{2k}$
is asymptotically equal to
\[
\IE_{\beta,N}\Sigma_{N}^{2k}\approx\begin{cases}
\left(\frac{1-\left(1-\alpha\right)\beta}{1-\beta}\right)^{k}\,K_{\lambda,N_{\lambda}}^{k} & \text{if }\,\beta<1,\\
m\left(\beta\right)^{2k}\,K_{\lambda,N_{\lambda}}^{2k} & \text{if }\,\beta>1,
\end{cases}
\]
where $0<m\left(\beta\right)<1$ for $\beta>1$ is the constant from
Definition \ref{def:m_beta}.

Let the model be in the high temperature regime, i.e. $\beta<1$.
Then there is a positive constant $\boldsymbol{D}_{\textup{high}}$
such that for all $N\in\IN$
\[
\left|\IE_{\beta,N}\frac{\Sigma_{K_{N}}^{2k}}{K_{N}^{k}}-\left(\frac{1-\left(1-\alpha\right)\beta}{1-\beta}\right)^{k}\right|<\boldsymbol{D}_{\textup{high}}\frac{1}{\sqrt{K_{N}}}.
\]
Now let the model be in the low temperature regime, i.e. $\beta>1$.
Then there is a positive constant $\boldsymbol{D}_{\textup{low}}$
such that for all $N\in\IN$
\[
\left|\IE_{\beta,N}\frac{\Sigma_{K_{N}}^{2k}}{K_{N}^{2k}}-m\left(\beta\right)^{2k}\right|<\boldsymbol{D}_{\textup{low}}\frac{\left(\ln N\right)^{\frac{3}{2}}}{\sqrt{N}}.
\]
\end{prop}

\subsection{Special case of $K_{N}=2$ for the Pair Correlation Estimator}

To carry out an estimation based on pair correlations between votes,
we will need access to a sample of observations of at least two of
the votes, so we will first assume $K_{N}=2$ for all $N$. We next
need to know how the correlation $\textup{Corr}\left(X_{1},X_{2}\right)$
between the first two votes in the population behaves as $N\rightarrow\infty$.
First, we note that
\begin{align*}
\textup{Corr}\left(X_{1},X_{2}\right) & =\frac{\textup{Cov}\left(X_{1},X_{2}\right)}{\sqrt{\IV_{\beta,N}X_{1}}\sqrt{\IV_{\beta,N}X_{2}}}=\frac{\IE_{\beta,N}X_{1}X_{2}-\IE_{\beta,N}X_{1}\,\IE_{\beta,N}X_{2}}{\sqrt{\IE_{\beta,N}X_{1}^{2}-\left(\IE_{\beta,N}X_{1}\right)^{2}}\sqrt{\IE_{\beta,N}X_{2}^{2}-\left(\IE_{\beta,N}X_{2}\right)^{2}}}\\
 & =\frac{\IE_{\beta,N}X_{1}X_{2}-0}{\sqrt{1-0}\sqrt{1-0}}=\IE_{\beta,N}X_{1}X_{2}.
\end{align*}
Proposition \ref{prop:appr_correlations} gives the asymptotic behaviour
of the expression $\IE_{\beta,N}X_{1}X_{2}$.

Instead of giving a proof of Theorem \ref{thm:properties_gamma_K},
we develop the special case of having access to the voting behaviour
of two voters from each group, the smallest possible subset of votes
which yields information about the coupling constants in each group.
We define a statistic $P_{2}$ that will allow us to define an estimator
for $\beta$ based on the empirical correlation between the first
two votes.
\begin{defn}
\label{def:P2}Let $n,N\in\IN$. Then we define, for all $\left(x^{(1)},\ldots,x^{(n)}\right)\in\Omega_{2}^{n}$,
\[
P_{2}\left(x^{(1)},\ldots,x^{(n)}\right)\coloneq\frac{1}{n}\sum_{t=1}^{n}x_{1}^{(t)}x_{2}^{(t)}.
\]
\end{defn}

Finally, we define the estimator $\hat{\gamma}_{2}$ based on the
pair correlation between the votes $X_{1}$ and $X_{2}$.
\begin{defn}
\label{def:pair_2}Let $\beta\geq0$ and $n,N\in\IN$. Let $0\leq\beta_{1}<1<\beta_{2}$
be constants satisfying the condition (\ref{eq:separation}), and
consider the intervals from Definition \ref{def:intervals}. Then
we define the \emph{estimator} $\hat{\gamma}_{2}:P_{2}^{-1}\left(J_{h}\cup J{}_{l}\right)\rightarrow\left[-\infty,\infty\right]$
\emph{based on the correlation between the first two vote}s, for all
$\left(x^{(1)},\ldots,x^{(n)}\right)\in\Omega_{2}^{n}$, by
\begin{enumerate}
\item If $P_{2}\left(x^{(1)},\ldots,x^{(n)}\right)\in J{}_{h}$, then
\[
\hat{\gamma}_{2}\coloneq\begin{cases}
\frac{NP_{2}\left(x^{(1)},\ldots,x^{(n)}\right)}{NP_{2}\left(x^{(1)},\ldots,x^{(n)}\right)+1} & \textup{if }P_{2}\left(x^{(1)},\ldots,x^{(n)}\right)>-\frac{1}{N},\\
-\infty & \textup{if }P_{2}\left(x^{(1)},\ldots,x^{(n)}\right)\leq-\frac{1}{N}.
\end{cases}
\]
\item If $P_{2}\left(x^{(1)},\ldots,x^{(n)}\right)\in J{}_{l}$, then $\hat{\gamma}_{2}>1$
is given by the unique value for which $P_{2}\left(x^{(1)},\ldots,x^{(n)}\right)=m\left(\hat{\gamma}_{2}\right)^{2}$
is satisfied.
\item If $P_{2}\left(x^{(1)},\ldots,x^{(n)}\right)\in J{}_{c}$, then we
say there is insufficient evidence in the sample to conclude that
$\beta$ is significantly different from 1, and $\hat{\gamma}_{2}$
is undefined.
\end{enumerate}
\end{defn}

In preparation of Theorem \ref{thm:properties_gamma_2} about this
estimator, we define the function that maps $\beta$ to the approximation
of the expectation $\IE_{\beta,N}X_{1}X_{2}$ given in Definition
\ref{def:gamma_tilde} and analyse its properties.
\begin{defn}
\label{def:fn_EXX}Let for each $N\in\IN$ the function $\varrho_{N}:\left[-\infty,\infty\right]\backslash I_{c}\rightarrow\IR$
be defined by
\[
\varrho_{N}\left(\beta\right)\coloneq\frac{\beta}{1-\beta}\frac{1}{N},\;\beta\leq\beta_{1},\quad\varrho_{N}\left(\beta\right)\coloneq m\left(\beta\right)^{2},\;\beta\geq\beta_{2},\quad\textup{and}\quad\varrho_{N}\left(-\infty\right)\coloneq-\frac{1}{N}.
\]
\end{defn}

\begin{lem}
\label{lem:fn_EXX}$\varrho_{N}$ is strictly increasing and continuously
differentiable on $\IR\backslash I'_{c}$.\\
The inverse function $\varrho_{N}^{-1}:\left[-\frac{1}{N},1\right]\backslash\left(\frac{\beta_{1}}{1-\beta_{1}}\frac{1}{N},m\left(\beta_{2}\right)^{2}\right)\rightarrow\left[-\infty,\infty\right]\backslash I'_{c}$
exists and is strictly increasing and continuously differentiable
on $\left(-\frac{1}{N},\frac{\beta_{1}}{1-\beta_{1}}\frac{1}{N}\right]\cup\left[m\left(\beta_{2}\right)^{2},1\right)$.

The derivative has the value
\[
\left(\varrho_{N}^{-1}\right)'\left(y\right)=\frac{1}{m'\left(m^{-1}\left(\sqrt{y}\right)\right)}\frac{1}{2\sqrt{y}},\quad y\in\left[m\left(\beta_{2}\right)^{2},1\right).
\]
\end{lem}

\begin{proof}
The proof of the statements for $\beta>1$ is the same as that for
Lemma 26 in Article 2, and the proof of the statements for $\beta<1$
is very similar.
\end{proof}
\begin{defn}
The next theorem sums up the properties of $\hat{\gamma}_{2}$.
\end{defn}

\begin{thm}
\label{thm:properties_gamma_2}The following statements hold:
\begin{enumerate}
\item For fixed $N\in\IN$, $\hat{\gamma_{2}}\xrightarrow[n\rightarrow\infty]{\textup{p}}\tilde{\gamma}_{N}$.
\item $\left|\tilde{\gamma}_{N}-\beta\right|\xrightarrow[N\rightarrow\infty]{}0$.
\item For fixed $N\in\IN$, $\sqrt{n}\left(\hat{\gamma_{2}}-\tilde{\gamma}_{N}\right)\xrightarrow[n\rightarrow\infty]{\textup{d}}\mathcal{N}\left(0,\sigma_{N}^{2}\right)$,
and
\[
\sigma_{N}^{2}\approx\begin{cases}
\left(1-\beta\right)^{4}N^{2} & \textup{if }\beta\in I{}_{h},\\
\frac{1-m\left(\beta\right)^{2}}{\left(2m\left(\beta\right)m'\left(\beta\right)\right)^{2}} & \textup{if }\beta\in I{}_{l}.
\end{cases}
\]
\end{enumerate}
\end{thm}

\begin{proof}
We prove each statement in turn.
\begin{enumerate}
\item Since $\IE\,P_{2}=\IE_{\beta,N}X_{1}X_{2}$ and $X_{1}X_{2}$ is a
bounded random variable, the weak law of large numbers says
\begin{equation}
P_{2}\xrightarrow[n\rightarrow\infty]{\textup{p}}\IE_{\beta,N}X_{1}X_{2}=\begin{cases}
\frac{\tilde{\gamma}_{N}}{1-\tilde{\gamma}_{N}}\frac{1}{N} & \textup{if }\beta\in I{}_{h},\\
m\left(\tilde{\gamma}_{N}\right)^{2} & \textup{if }\beta\in I{}_{l}.
\end{cases}\label{eq:T_WLLN-1}
\end{equation}
For a fixed sample $\left(x^{(1)},\ldots,x^{(n)}\right)\in\Omega_{2}^{n}$,
by Definition \ref{def:pair_2}, $\hat{\gamma}_{2}$ satisfies
\begin{equation}
P_{2}\left(x^{(1)},\ldots,x^{(n)}\right)=\begin{cases}
\frac{\hat{\gamma_{2}}\left(x^{(1)},\ldots,x^{(n)}\right)}{1-\hat{\gamma_{2}}\left(x^{(1)},\ldots,x^{(n)}\right)}\frac{1}{N} & \textup{if }P_{2}\left(x^{(1)},\ldots,x^{(n)}\right)\in J{}_{h},\\
m\left(\hat{\gamma_{2}}\left(x^{(1)},\ldots,x^{(n)}\right)\right)^{2} & \textup{if }P_{2}\left(x^{(1)},\ldots,x^{(n)}\right)\in J{}_{l}.
\end{cases}\label{eq:T_beta_inf-1}
\end{equation}
Let $\beta\in I_{h}$. We show $\frac{\hat{\gamma_{2}}}{1-\hat{\gamma_{2}}}\frac{1}{N}\xrightarrow[n\rightarrow\infty]{\textup{p}}\frac{\tilde{\gamma}_{N}}{1-\tilde{\gamma}_{N}}\frac{1}{N}$.
The statement then follows by Theorem \ref{thm:cont_mapping}.\\
Let $\varepsilon>0$. Recall the intervals $J_{h}$ and $J_{l}$ from
Definition \ref{def:intervals}. Assume without loss of generality
that $\varepsilon$ is small enough that $\left|P_{2}-\IE_{\beta,N}X_{1}X_{2}\right|\leq\varepsilon$
implies $P_{2}\in J{}_{h}$ (cf. Remark \ref{rem:intervals}). We
write
\begin{align*}
\IP\left\{ \left|\frac{\hat{\gamma_{2}}}{1-\hat{\gamma_{2}}}\frac{1}{N}-\frac{\tilde{\gamma}_{N}}{1-\tilde{\gamma}_{N}}\frac{1}{N}\right|>\varepsilon\right\}  & =\IP\left\{ \left|\frac{\hat{\gamma_{2}}}{1-\hat{\gamma_{2}}}\frac{1}{N}-\frac{\tilde{\gamma}_{N}}{1-\tilde{\gamma}_{N}}\frac{1}{N}\right|>\varepsilon,\left|P_{2}-\IE_{\beta,N}X_{1}X_{2}\right|\leq\varepsilon\right\} \\
 & \quad+\IP\left\{ \left|\frac{\hat{\gamma_{2}}}{1-\hat{\gamma_{2}}}\frac{1}{N}-\frac{\tilde{\gamma}_{N}}{1-\tilde{\gamma}_{N}}\frac{1}{N}\right|>\varepsilon,\left|P_{2}-\IE_{\beta,N}X_{1}X_{2}\right|>\varepsilon\right\} .
\end{align*}
The latter summand is smaller than or equal to
\[
\IP\left\{ \left|P_{2}-\IE_{\beta,N}X_{1}X_{2}\right|>\varepsilon\right\} 
\]
which converges to $0$ by (\ref{eq:T_WLLN-1}).\\
We turn to the first summand. By Definition \ref{def:gamma_tilde},
\[
\frac{\tilde{\gamma}_{N}}{1-\tilde{\gamma}_{N}}\frac{1}{N}=\IE_{\beta,N}X_{1}X_{2}.
\]
By Definition \ref{def:pair_K} and under the assumption $P_{2}\in J{}_{h}$,
we have
\[
\frac{\hat{\gamma_{2}}}{1-\hat{\gamma_{2}}}\frac{1}{N}=P_{2}.
\]
Combining the last two displays, we obtain
\begin{align*}
\IP\left\{ \left|\frac{\hat{\gamma_{2}}}{1-\hat{\gamma_{2}}}\frac{1}{N}-\frac{\tilde{\gamma}_{N}}{1-\tilde{\gamma}_{N}}\frac{1}{N}\right|>\varepsilon,\left|P_{2}-\IE_{\beta,N}X_{1}X_{2}\right|\leq\varepsilon\right\}  & =\IP\left\{ \left|P_{2}-\IE_{\beta,N}X_{1}X_{2}\right|>\varepsilon,\left|P_{2}-\IE_{\beta,N}X_{1}X_{2}\right|\leq\varepsilon\right\} \\
 & =\IP\,\emptyset=0.
\end{align*}
We have therefore proved
\[
\IP\left\{ \left|\frac{\hat{\gamma_{2}}}{1-\hat{\gamma_{2}}}\frac{1}{N}-\frac{\tilde{\gamma}_{N}}{1-\tilde{\gamma}_{N}}\frac{1}{N}\right|>\varepsilon\right\} \xrightarrow[n\rightarrow\infty]{}0
\]
and thus $\frac{\hat{\gamma_{2}}}{1-\hat{\gamma_{2}}}\frac{1}{N}\xrightarrow[n\rightarrow\infty]{\textup{p}}\frac{\tilde{\gamma}_{N}}{1-\tilde{\gamma}_{N}}\frac{1}{N}$.\\
The case $\beta\in I{}_{l}$ is treated analogously: we first show
$m\left(\hat{\gamma_{2}}\right)^{2}N^{2}\xrightarrow[n\rightarrow\infty]{\textup{p}}m\left(\tilde{\gamma}_{N}\right)^{2}N^{2}$,
and then use that the function $\beta\in\left(1,\infty\right)\mapsto m\left(\beta\right)\in\left(0,1\right)$
is strictly increasing by Lemma \ref{lem:m_beta_increasing}. Thus,
$\hat{\gamma_{2}}\xrightarrow[n\rightarrow\infty]{\textup{p}}\tilde{\gamma}_{N}$
follows from Lemma \ref{lem:fn_EXX}, Theorem \ref{thm:cont_mapping},
and $m\left(\hat{\gamma_{2}}\right)^{2}N^{2}\xrightarrow[n\rightarrow\infty]{\textup{p}}m\left(\tilde{\gamma}_{N}\right)^{2}N^{2}$.
\item By Proposition \ref{prop:appr_correlations}, we have
\begin{align*}
N\,\IE_{\beta,N}X_{1}X_{2} & \xrightarrow[N\rightarrow\infty]{}\frac{\beta}{1-\beta}\quad\textup{if }\beta\in I{}_{h},\\
\IE_{\beta,N}X_{1}X_{2} & \xrightarrow[N\rightarrow\infty]{}m\left(\beta\right)^{2}\quad\textup{if }\beta\in I{}_{l}.
\end{align*}
The mapping $\beta\mapsto\IE_{\beta,N}X_{1}X_{2}$ is continuous.
Hence, for $\beta\in I{}_{h}$,
\[
\frac{\tilde{\gamma}_{N}}{1-\tilde{\gamma}_{N}}=N\,\IE_{\beta,N}X_{1}X_{2}\xrightarrow[N\rightarrow\infty]{}\frac{\beta}{1-\beta},
\]
which is equivalent to $\tilde{\gamma}_{N}\xrightarrow[N\rightarrow\infty]{}\beta$.\\
Recall Definition \ref{def:m_beta} of $\beta\mapsto m\left(\beta\right)$
and its continuity which follows from Lemma \ref{lem:m_beta_diff}.
For $\beta\in I{}_{l}$,
\[
m\left(\tilde{\gamma}_{N}\right)^{2}=\IE_{\beta,N}X_{1}X_{2}\xrightarrow[N\rightarrow\infty]{}m\left(\beta\right)^{2}.
\]
The function $\beta\in\left(1,\infty\right)\mapsto m\left(\beta\right)\in\left(0,1\right)$
is strictly increasing by Lemma \ref{lem:m_beta_increasing}. Thus,
$\tilde{\gamma}_{N}\xrightarrow[N\rightarrow\infty]{}\beta$ follows
from Lemma \ref{lem:m_beta_diff} and $m\left(\tilde{\gamma}_{N}\right)^{2}\xrightarrow[N\rightarrow\infty]{}m\left(\beta\right)^{2}$.
\item We first show the result for $\beta\in I{}_{h}$. We once again use
Definitions \ref{def:P2}, \ref{def:pair_2}, and \ref{def:gamma_tilde}
of $P_{2}$, $\hat{\gamma_{2}}$, and $\tilde{\gamma}_{N}$, and the
law of large numbers (\ref{eq:T_WLLN-1}). In addition to the law
of large numbers, we also note that
\[
P_{2}-\IE\,P_{2}=\frac{1}{n}\sum_{t=1}^{n}\left(X_{1}^{(t)}X_{2}^{(t)}-\IE\,P_{2}\right)
\]
is the sum of i.i.d. random variables with
\[
\IE\left(X_{1}^{(t)}X_{2}^{(t)}-\IE\,P_{2}\right)=0,\qquad\IV\left[\left(X_{1}^{(t)}X_{2}^{(t)}-\IE P_{2}\right)\right]=\IV_{\beta,N}X_{1}X_{2}
\]
for all $t\in\IN_{n}$. The central limit theorem yields
\begin{equation}
\sqrt{n}\left(P_{2}-\IE\,P_{2}\right)\xrightarrow[n\rightarrow\infty]{\textup{d}}\mathcal{N}\left(0,\IV_{\beta,N}X_{1}X_{2}\right).\label{eq:CLT_T-1}
\end{equation}
For a fixed sample $\left(x^{(1)},\ldots,x^{(n)}\right)\in\Omega_{2}^{n}$
with $P_{2}\left(x^{(1)},\ldots,x^{(n)}\right)\in J{}_{h}$, by Definition
\ref{def:pair_2}, $\hat{\gamma_{2}}$ satisfies
\[
P_{2}\left(x^{(1)},\ldots,x^{(n)}\right)=\frac{\hat{\gamma_{2}}\left(x^{(1)},\ldots,x^{(n)}\right)}{1-\hat{\gamma_{2}}\left(x^{(1)},\ldots,x^{(n)}\right)}\frac{1}{N}
\]
and by Definition \ref{def:gamma_tilde}
\[
\IE_{\beta,N}P_{2}=\IE_{\beta,N}X_{1}X_{2}=\frac{\tilde{\gamma}_{N}}{1-\tilde{\gamma}_{N}}\frac{1}{N}.
\]
Joining the last two display, we obtain
\[
P_{2}-\IE\,P_{2}=\frac{\hat{\gamma_{2}}}{1-\hat{\gamma_{2}}}\frac{1}{N}-\frac{\tilde{\gamma}_{N}}{1-\tilde{\gamma}_{N}}\frac{1}{N}.
\]
Next, we calculate
\[
\frac{\hat{\gamma_{2}}}{1-\hat{\gamma_{2}}}-\frac{\tilde{\gamma}_{N}}{1-\tilde{\gamma}_{N}}=\frac{\hat{\gamma_{2}}-\tilde{\gamma}_{N}}{\left(1-\hat{\gamma_{2}}\right)\left(1-\tilde{\gamma}_{N}\right)}.
\]
The last display together with (\ref{eq:CLT_T-1}) and Theorem \ref{thm:slutsky}
implies
\begin{equation}
\frac{\hat{\gamma_{2}}-\tilde{\gamma}_{N}}{\left(1-\hat{\gamma_{2}}\right)\left(1-\tilde{\gamma}_{N}\right)}=\sqrt{n}N\left(P_{2}-\IE\,P_{2}\right)\xrightarrow[n\rightarrow\infty]{\textup{d}}\mathcal{N}\left(0,N^{2}\,\IV_{\beta,N}X_{1}X_{2}\right).\label{eq:CLT_2-1}
\end{equation}
We define the following sequences of random variables:
\[
U_{n}\coloneq\left(\hat{\gamma_{2}}-\tilde{\gamma}_{N}\right)\frac{1-\tilde{\gamma}_{N}}{1-\hat{\gamma_{2}}},\qquad V_{n}\coloneq\hat{\gamma_{2}}-\tilde{\gamma}_{N},\qquad n\in\IN.
\]
Then (\ref{eq:CLT_2-1}) is equivalent to
\[
\frac{\sqrt{n}}{\left(1-\tilde{\gamma}_{N}\right)^{2}}U_{n}\xrightarrow[n\rightarrow\infty]{\textup{d}}\mathcal{N}\left(0,N^{2}\,\IV_{\beta,N}X_{1}X_{2}\right).
\]
It follows with Theorem \ref{thm:slutsky} that
\begin{equation}
\sqrt{n}\,U_{n}\xrightarrow[n\rightarrow\infty]{\textup{d}}\mathcal{N}\left(0,\left(1-\tilde{\gamma}_{N}\right)^{4}N^{2}\,\IV_{\beta,N}X_{1}X_{2}\right).\label{eq:CLT_3-1}
\end{equation}
Next, we show
\begin{equation}
\frac{1-\hat{\gamma_{2}}}{1-\tilde{\gamma}_{N}}\xrightarrow[n\rightarrow\infty]{\textup{p}}1.\label{eq:frac_conv_p-1}
\end{equation}
We have
\[
\left|\frac{1-\hat{\gamma_{2}}}{1-\tilde{\gamma}_{N}}-1\right|=\frac{\left|\hat{\gamma_{2}}-\tilde{\gamma}_{N}\right|}{1-\tilde{\gamma}_{N}}\xrightarrow[n\rightarrow\infty]{\textup{p}}0,
\]
where the convergence in probability follows from statement 1 of this
theorem and Theorem \ref{thm:slutsky}.\\
Since
\[
\frac{V_{n}}{U_{n}}=\frac{1-\hat{\gamma_{2}}}{1-\tilde{\gamma}_{N}},
\]
we arrive at
\[
\sqrt{n}\,V_{n}=\sqrt{n}\,U_{n}\,\frac{V_{n}}{U_{n}}\xrightarrow[n\rightarrow\infty]{\textup{d}}\mathcal{N}\left(0,\left(1-\tilde{\gamma}_{N}\right)^{4}N^{2}\,\IV_{\beta,N}\,X_{1}X_{2}\right)
\]
using (\ref{eq:CLT_3-1}), (\ref{eq:frac_conv_p-1}), and Theorem
\ref{thm:slutsky}.\\
We use Proposition \ref{prop:appr_correlations}:
\begin{align*}
\IV_{\beta,N}\,X_{1}X_{2} & =\IE_{\beta,N}\left(X_{1}X_{2}\right)^{2}-\left(\IE_{\beta,N}X_{1}X_{2}\right)^{2}=1-\left(\frac{\tilde{\gamma}_{N}}{1-\tilde{\gamma}_{N}}\frac{1}{N}\right)^{2},
\end{align*}
and the claim
\[
\left(1-\tilde{\gamma}_{N}\right)^{4}N^{2}\,\IV_{\beta,N}X_{1}X_{2}\approx\left(1-\beta\right)^{4}N^{2}
\]
follows.\\
Now let $\beta\in I{}_{l}$. The law of large numbers (\ref{eq:T_WLLN-1})
and the central limit theorem (\ref{eq:CLT_T-1}) hold, and for a
fixed sample $\left(x^{(1)},\ldots,x^{(n)}\right)\in\Omega_{2}^{n}$
with $P_{2}\left(x^{(1)},\ldots,x^{(n)}\right)\in J{}_{l}$, by Definition
\ref{def:pair_2}, $\hat{\gamma}_{2}$ satisfies
\[
P_{2}\left(x^{(1)},\ldots,x^{(n)}\right)=m\left(\hat{\gamma}_{2}\left(x^{(1)},\ldots,x^{(n)}\right)\right)^{2}
\]
and by Definition \ref{def:gamma_tilde}
\[
\IE_{\beta,N}P_{2}=\IE_{\beta,N}X_{1}X_{2}=m\left(\tilde{\gamma}_{N}\right)^{2}.
\]
Joining the last two displays, we obtain
\[
P_{2}-\IE_{\beta,N}P_{2}=m\left(\hat{\gamma}_{2}\right)^{2}-m\left(\tilde{\gamma}_{N}\right)^{2}.
\]
By (\ref{eq:CLT_T-1}), we have
\[
\sqrt{n}\left(m\left(\hat{\gamma}_{2}\right)^{2}-m\left(\tilde{\gamma}_{N}\right)^{2}\right)\xrightarrow[n\rightarrow\infty]{\textup{d}}\mathcal{N}\left(0,\IV_{\beta,N}\,X_{1}X_{2}\right).
\]
We define
\[
W_{n}\coloneq\sqrt{n}\left(m\left(\hat{\gamma}_{2}\right)^{2}-m\left(\tilde{\gamma}_{N}\right)^{2}\right)\,\II_{\left\{ m\left(\beta_{2}\right)^{2}<m\left(\hat{\gamma}_{2}\right)^{2}<\frac{m\left(\tilde{\gamma}_{N}\right)^{2}+1}{2}\right\} },\quad n\in\IN,
\]
and apply Lemma \ref{lem:conv_restr_sequence} to the sequence $Y_{n}\coloneq\sqrt{n}\left(m\left(\hat{\gamma}_{2}\right)^{2}-m\left(\tilde{\gamma}_{N}\right)^{2}\right)$
with $K\coloneq\left(m\left(\beta_{2}\right)^{2},\frac{m\left(\tilde{\gamma}_{N}\right)^{2}+1}{2}\right)^{c}$,
$B_{n}\coloneq\left\{ Y_{n}\,|\,m\left(\hat{\gamma}_{2}\right)^{2}\in K\right\} $
for each $n\in\IN$, and $\nu\coloneq\mathcal{N}\left(0,\IV_{\beta,N}\,X_{1}X_{2}\right)$.
We set $M_{n}\coloneq\sqrt{n}$ and $B\coloneq\bigcup_{n\in\IN}B_{n}$.
By the large deviations principle, we have due to $m\left(\beta_{2}\right)^{2}<\IE_{\beta,N}X_{1}X_{2}<1$,
for the closed set $K$,
\[
\IP\left\{ Y_{n}\in B\right\} \leq2\exp\left(-n\inf_{x\in K}\Lambda_{S^{2}}^{*}\left(x\right)\right),\quad n\in\IN.
\]
Therefore, $\IP\left\{ Y_{n}\in B\right\} =o\left(\frac{1}{M_{n}}\right)$
holds, and we can apply Lemma \ref{lem:conv_restr_sequence} to conclude
\[
W_{n}\xrightarrow[n\rightarrow\infty]{\textup{d}}\mathcal{N}\left(0,\IV_{\beta,N}X_{1}X_{2}\right).
\]
We apply Theorem \ref{thm:delta_method} to the sequence $W_{n}$.
Set $D\coloneq\left[m\left(\beta_{2}\right)^{2},\frac{m\left(\tilde{\gamma}_{N}\right)^{2}+1}{2}\right]$
and $f:D\rightarrow\IR$
\[
f\left(y\right)\coloneq\varrho_{N}^{^{-1}}\left(y\right)=m^{-1}\left(\sqrt{y}\right),\quad y\in D.
\]
$f$ is continuously differentiable and strictly positive on the compact
set $D$. We have $\mu=\IE_{\beta,N}m\left(\hat{\gamma}_{2}\right)^{2}=m\left(\tilde{\gamma}_{N}\right)^{2}$
and $\sigma^{2}=\IV_{\beta,N}\,X_{1}X_{2}$. Then
\begin{align*}
 & \quad\sqrt{n}\left(\hat{\gamma}_{2}-\tilde{\gamma}_{N}\right)\,\II_{\left\{ m\left(\beta_{2}\right)^{2}<m\left(\hat{\gamma}_{2}\right)^{2}<\frac{m\left(\tilde{\gamma}_{N}\right)^{2}+1}{2}\right\} }\\
 & =\sqrt{n}\left(\hat{\gamma}_{2}-\tilde{\gamma}_{N}\right)\,\II_{\left\{ m\left(\beta_{2}\right)^{2}<m\left(\hat{\gamma}_{2}\right)^{2}<\frac{m\left(\tilde{\gamma}_{N}\right)^{2}+1}{2}\right\} }+\sqrt{n}\left(\tilde{\gamma}_{N}-\tilde{\gamma}_{N}\right)\,\II_{\left\{ m\left(\beta_{2}\right)^{2}<m\left(\hat{\gamma}_{2}\right)^{2}<\frac{m\left(\tilde{\gamma}_{N}\right)^{2}+1}{2}\right\} ^{c}}\\
 & =\sqrt{n}\left(\left(\hat{\gamma}_{2}\,\II_{\left\{ m\left(\beta_{2}\right)^{2}<m\left(\hat{\gamma}_{2}\right)^{2}<\frac{m\left(\tilde{\gamma}_{N}\right)^{2}+1}{2}\right\} }+\tilde{\gamma}_{N}\,\II_{\left\{ m\left(\beta_{2}\right)^{2}<m\left(\hat{\gamma}_{2}\right)^{2}<\frac{m\left(\tilde{\gamma}_{N}\right)^{2}+1}{2}\right\} ^{c}}\right)-\tilde{\gamma}_{N}\right)\\
 & =\sqrt{n}\left(f\left(m\left(\hat{\gamma}_{2}\right)^{2}\,\II_{\left\{ m\left(\beta_{2}\right)^{2}<m\left(\hat{\gamma}_{2}\right)^{2}<\frac{m\left(\tilde{\gamma}_{N}\right)^{2}+1}{2}\right\} }+m\left(\tilde{\gamma}_{N}\right)^{2}\,\II_{\left\{ m\left(\beta_{2}\right)^{2}<m\left(\hat{\gamma}_{2}\right)^{2}<\frac{m\left(\tilde{\gamma}_{N}\right)^{2}+1}{2}\right\} ^{c}}\right)-f\left(m\left(\tilde{\gamma}_{N}\right)^{2}\right)\right)\\
 & \xrightarrow[n\rightarrow\infty]{\textup{d}}\mathcal{N}\left(0,\left(f'\left(\mu\right)\right)^{2}\sigma^{2}\right)
\end{align*}
holds. We apply Lemma \ref{lem:conv_restr_sequence} once more to
conclude that
\[
\sqrt{n}\left(\hat{\gamma}_{2}-\tilde{\gamma}_{N}\right)\xrightarrow[n\rightarrow\infty]{\textup{d}}\mathcal{N}\left(0,\left(f'\left(\mu\right)\right)^{2}\sigma^{2}\right)
\]
is satisfied. By Lemma \ref{lem:fn_EXX},
\[
f'\left(y\right)=\left(\varrho_{N}^{^{-1}}\right)'\left(y\right)=\frac{1}{m'\left(m^{-1}\left(\sqrt{y}\right)\right)}\frac{1}{2\sqrt{y}},
\]
so, taking into account $\mu=m\left(\tilde{\gamma}_{N}\right)^{2}$,
\begin{align*}
\left(f'\left(\mu\right)\right)^{2}\sigma^{2} & =\left[\frac{1}{m'\left(m^{-1}\left(m\left(\tilde{\gamma}_{N}\right)\right)\right)}\frac{1}{2m\left(\tilde{\gamma}_{N}\right)}\right]^{2}\\
 & =\left[\frac{1}{m'\left(\tilde{\gamma}_{N}\right)}\frac{1}{2m\left(\tilde{\gamma}_{N}\right)}\right]^{2},
\end{align*}
and the statement concerning the limiting variance follows.
\end{enumerate}
\end{proof}

\subsection{Proof of Theorem \ref{thm:properties_gamma_K}}
\begin{proof}[Proof of Theorem \ref{thm:properties_gamma_K}]
The proof of this theorem is completely analogous to that of Theorem
\ref{thm:properties_gamma_2}. We only show the statements concerning
the limiting variance of the estimator $\hat{\gamma}_{K_{N}}$.

Let $\beta\in I{}_{h}$. We arrive at the statements
\[
\sqrt{n}\left(P_{K_{N}}-\IE P_{K_{N}}\right)\xrightarrow[n\rightarrow\infty]{\textup{d}}\mathcal{N}\left(0,\frac{1}{\left(K_{N}\left(K_{N}-1\right)\right)^{2}}\,\IV_{\beta,N}\sum_{1\leq i,j\leq K_{N},i\neq j}X_{i}X_{j}\right)
\]
and then
\[
\sqrt{n}\left(\hat{\gamma}_{K_{N}}-\tilde{\gamma}_{N}\right)\xrightarrow[n\rightarrow\infty]{\textup{d}}\mathcal{N}\left(0,\frac{\left(1-\tilde{\gamma}_{N}\right)^{4}N^{2}}{\left(K_{N}\left(K_{N}-1\right)\right)^{2}}\,\IV_{\beta,N}\sum_{1\leq i,j\leq K_{N},i\neq j}X_{i}X_{j}\right).
\]
We have
\[
\sum_{1\leq i,j\leq K_{N},i\neq j}X_{i}X_{j}=\left(\sum_{i=1}^{K_{N}}X_{i}\right)^{2}-\sum_{i=1}^{K_{N}}X_{i}^{2}=\Sigma_{K_{N}}^{2}-K_{N},
\]
and therefore
\begin{align*}
\IV_{\beta,N}\sum_{1\leq i,j\leq K_{N},i\neq j}X_{i}X_{j} & =\IE_{\beta,N}\left(\Sigma_{K_{N}}^{2}-K_{N}\right)^{2}-\left(\IE_{\beta,N}\Sigma_{K_{N}}^{2}-K_{N}\right)^{2}\\
 & =\IE_{\beta,N}\Sigma_{K_{N}}^{4}-\left(\IE_{\beta,N}\Sigma_{K_{N}}^{2}\right)^{2}\\
 & =\IV_{\beta,N}\Sigma_{K_{N}}^{2}.
\end{align*}
The statement concerning the limiting variance for $\beta\in I{}_{h}$
and $\alpha>0$ then follows from Definition \ref{def:alpha} and
Proposition \ref{prop:expec_Sigma}.

Let $\beta\in I{}_{l}$. We have
\[
\sqrt{n}\left(m\left(\hat{\gamma}_{K_{N}}\right)^{2}-m\left(\tilde{\gamma}_{N}\right)^{2}\right)\xrightarrow[n\rightarrow\infty]{\textup{d}}\mathcal{N}\left(0,\frac{1}{\left(K_{N}\left(K_{N}-1\right)\right)^{2}}\,\IV_{\beta,N}\sum_{1\leq i,j\leq K_{N},i\neq j}X_{i}X_{j}\right).
\]
We apply Theorem \ref{thm:delta_method} with $D\coloneq\left[m\left(\beta_{2}\right)^{2},\frac{m\left(\tilde{\gamma}_{N}\right)^{2}+1}{2}\right]$,
$\mu=m\left(\tilde{\gamma}_{N}\right)^{2}$, $\sigma^{2}=\frac{1}{\left(K_{N}\left(K_{N}-1\right)\right)^{2}}\,\IV_{\beta,N}\sum_{1\leq i,j\leq K_{N},i\neq j}X_{i}X_{j}$,
and $f:D\rightarrow\IR$
\[
f\left(y\right)\coloneq\varrho_{N}^{^{-1}}\left(y\right)=m^{-1}\left(\sqrt{y}\right),\quad y\in D.
\]
$f$ is continuously differentiable and strictly positive on the compact
set $D$. Therefore,
\begin{align*}
\sqrt{n}\left(\hat{\gamma}_{K_{N}}-\tilde{\gamma}_{N}\right)\II_{\left\{ m\left(\beta_{2}\right)^{2}<m\left(\hat{\gamma}_{K_{N}}\right)^{2}<\frac{m\left(\tilde{\gamma}_{N}\right)^{2}+1}{2}\right\} } & \xrightarrow[n\rightarrow\infty]{\textup{d}}\mathcal{N}\left(0,\left(f'\left(\mu\right)\right)^{2}\sigma^{2}\right)
\end{align*}
according to the same arguments as above. An application of Lemma
\ref{lem:conv_restr_sequence} yields
\begin{align*}
\sqrt{n}\left(\hat{\gamma}_{K_{N}}-\tilde{\gamma}_{N}\right) & \xrightarrow[n\rightarrow\infty]{\textup{d}}\mathcal{N}\left(0,\left(f'\left(\mu\right)\right)^{2}\sigma^{2}\right)
\end{align*}
The limiting variance is the product of
\[
\left(f'\left(\mu\right)\right)^{2}=\frac{1}{\left(2m\left(\beta\right)m'\left(\beta\right)\right)^{2}}
\]
and
\begin{align*}
\sigma^{2} & =\frac{1}{\left(K_{N}\left(K_{N}-1\right)\right)^{2}}\,\IV_{\beta,N}\sum_{1\leq i,j\leq K_{N},i\neq j}X_{i}X_{j}\\
 & =\frac{1}{\left(K_{N}\left(K_{N}-1\right)\right)^{2}}\,\IV_{\beta,N}\Sigma_{K_{N}}^{2}\\
 & =\left(\frac{K_{N}}{K_{N}-1}\right)^{2}\,\IV_{\beta,N}\left(\frac{\Sigma_{K_{N}}}{K_{N}}\right)^{2}.
\end{align*}
\end{proof}

\subsection{Proof of Theorem \ref{thm:asymp_equiv}}

Prior to proving Theorem \ref{thm:asymp_equiv}, we analyse some basic
properties of the ranges of the random variables $\hat{\gamma}_{K_{N}}$
and $\hat{\zeta}_{K_{N}}$. This will help in understanding the reason
why the statements in Theorem \ref{thm:asymp_equiv} are restricted
to certain subsets of the respective ranges.
\begin{lem}
\label{lem:minus_infty}Whereas for the estimator $\hat{\zeta}_{K_{N}}$
\[
\hat{\zeta}_{K_{N}}\left(x\right)=-\infty\quad\iff T_{K_{N}}\left(x\right)\leq K_{N}\left(1-\alpha\right),
\]
we have for the estimator $\hat{\gamma}_{K_{N}}$
\[
\hat{\gamma}_{K_{N}}\left(x\right)=-\infty\quad\iff T_{K_{N}}\left(x\right)\leq K_{N}\left(1-\frac{K_{N}-1}{N}\right)
\]
for all $x\in\Omega_{K_{N}}^{n}$.
\end{lem}

\begin{proof}
Recall Definitions \ref{def:PKN}, \ref{def:TKN}, \ref{def:pair_K}
and \ref{def:zeta_tilde}. We calculate for all $x\in\Omega_{K_{N}}^{n}$,
\begin{align*}
\hat{\gamma}_{K_{N}}\left(x\right) & =-\infty\quad\iff\\
P_{K_{N}}\left(x\right) & \leq-\frac{1}{N}\quad\iff\\
\frac{1}{n}\sum_{t=1}^{n}\frac{1}{K_{N}\left(K_{N}-1\right)}\sum_{1\leq i,j\leq K_{N},i\neq j}x_{i}^{(t)}x_{j}^{(t)} & \leq-\frac{1}{N}\quad\iff\\
\frac{1}{n}\sum_{t=1}^{n}\frac{1}{K_{N}\left(K_{N}-1\right)}\left[\left(\sum_{i=1}^{K_{N}}x_{i}^{\left(t\right)}\right)^{2}-\sum_{i=1}^{K_{N}}\left(x_{i}^{\left(t\right)}\right)^{2}\right] & \leq-\frac{1}{N}\quad\iff\\
\frac{1}{K_{N}\left(K_{N}-1\right)}\left(T_{K_{N}}\left(x\right)-K_{N}\right) & \leq-\frac{1}{N}\quad\iff\\
T_{K_{N}}\left(x\right) & \leq-\frac{K_{N}\left(K_{N}-1\right)}{N}+K_{N}\quad\iff\\
T_{K_{N}}\left(x\right) & \leq K_{N}\left(1-\frac{K_{N}-1}{N}\right).
\end{align*}
The statement
\begin{align*}
\hat{\zeta}_{K_{N}}\left(x\right)=-\infty & \quad\iff T_{K_{N}}\left(x\right)\leq K_{N}\left(1-\alpha\right)
\end{align*}
is part of Definition \ref{def:zeta_tilde}.
\end{proof}
The previous lemma and Definition \ref{def:alpha} imply that except
for some marginal cases, the two estimators $\hat{\gamma}_{K_{N}}$
and $\hat{\zeta}_{K_{N}}$ are equal to $-\infty$ for the same samples
$x\in\Omega_{K_{N}}^{n}$, provided that $N$ is large enough (and
hence $\left(K_{N}-1\right)/N$ is close enough to $\alpha$). It
should also be noted that the maximum likelihood estimator $\hat{\beta}_{N}^{\infty}$
(see Definition 21 in Article 2) based on a sample of votes by the
entire population takes finite values except for the two most extreme
types of samples:
\begin{enumerate}
\item If the entire sample $x\in\Omega_{N}^{n}$ consists of unanimous votes
(i.e. for all $t\in\IN_{n}$, $x_{i}^{\left(t\right)}=x_{j}^{\left(t\right)}$,
$i,j\in\IN_{N}$, holds), then $\hat{\beta}_{N}^{\infty}\left(x\right)=\infty$.
\item If the entire sample $x\in\Omega_{N}^{n}$ consists of polarised votes
(i.e. for all $t\in\IN_{n}$, $\left(\sum_{i=1}^{N}x_{i}^{\left(t\right)}\right)^{2}=\min\textup{Range}\left(S^{2}\right)$
holds), then $\hat{\beta}_{N}^{\infty}\left(x\right)=-\infty$.
\end{enumerate}
For all other samples $x$, $\hat{\beta}_{N}^{\infty}\left(x\right)\in\IR$
is satisfied. By Lemma \ref{lem:minus_infty}, this is not the case
for the estimators $\hat{\gamma}_{K_{N}}$ and $\hat{\zeta}_{K_{N}}$
which are based on a sample of votes by a subset of the population.
\begin{proof}
[Proof of Theorem \ref{thm:asymp_equiv}]Let $n\in\IN$, $\beta\in I_{h}$.
The statement
\[
\hat{\gamma}_{K_{N}}\left(x\right)=\hat{\zeta}_{K_{N}}\left(x\right)=-\infty,\quad x\in A_{N,n}\cap A'_{N,n}
\]
follows immediately from Lemma \ref{lem:minus_infty} and Definition
\ref{def:sets}.

We now show
\begin{equation}
\frac{K_{N}-1}{N}\frac{\hat{\gamma}_{K_{N}}\left(x\right)}{1-\hat{\gamma}_{K_{N}}\left(x\right)}=\frac{\alpha\,\hat{\zeta}_{K_{N}}\left(x\right)}{1-\hat{\zeta}_{K_{N}}\left(x\right)}\label{eq:two_est}
\end{equation}
for all $x\in H_{N,n}$. Using Definitions \ref{def:pair_K}, \ref{def:TKN},
and \ref{def:zeta^hat},
\begin{align*}
\frac{K_{N}}{N}\frac{\hat{\gamma}_{K_{N}}\left(x\right)}{1-\hat{\gamma}_{K_{N}}\left(x\right)} & =\frac{1}{n}\sum_{t=1}^{n}\frac{1}{K_{N}-1}\sum_{1\leq i,j\leq K_{N};\,i\neq j}x_{i}^{(t)}x_{j}^{(t)}\\
 & =\frac{1}{n}\sum_{t=1}^{n}\frac{1}{K_{N}-1}\left[\left(\sum_{i=1}^{K_{N}}x_{i}^{(t)}\right)^{2}-\sum_{i=1}^{K_{N}}\left(x_{i}^{(t)}\right)^{2}\right]\\
 & =\frac{1}{K_{N}-1}\left(T_{K_{N}}\left(x\right)-K_{N}\right)\\
 & =\frac{K_{N}}{K_{N}-1}\left(\frac{T_{K_{N}}\left(x\right)}{K_{N}}-1\right)\\
 & =\frac{K_{N}}{K_{N}-1}\left(\frac{1-\left(1-\alpha\right)\hat{\zeta}_{K_{N}}\left(x\right)}{1-\hat{\zeta}_{K_{N}}\left(x\right)}-1\right)\\
 & =\frac{K_{N}}{K_{N}-1}\frac{\alpha\,\hat{\zeta}_{K_{N}}\left(x\right)}{1-\hat{\zeta}_{K_{N}}\left(x\right)},
\end{align*}
and we obtain (\ref{eq:two_est}). Now we find an upper bound for
the difference
\begin{align*}
\left|\frac{\hat{\gamma}_{K_{N}}\left(x\right)}{1-\hat{\gamma}_{K_{N}}\left(x\right)}-\frac{\hat{\zeta}_{K_{N}}\left(x\right)}{1-\hat{\zeta}_{K_{N}}\left(x\right)}\right| & =\frac{1}{\alpha}\left|\frac{\alpha\,\hat{\gamma}_{K_{N}}\left(x\right)}{1-\hat{\gamma}_{K_{N}}\left(x\right)}-\frac{\alpha\,\hat{\zeta}_{K_{N}}\left(x\right)}{1-\hat{\zeta}_{K_{N}}\left(x\right)}\right|\\
 & \leq\frac{1}{\alpha}\left[\left|\frac{\alpha\,\hat{\gamma}_{K_{N}}\left(x\right)}{1-\hat{\gamma}_{K_{N}}\left(x\right)}-\frac{K_{N}-1}{N}\frac{\hat{\gamma}_{K_{N}}\left(x\right)}{1-\hat{\gamma}_{K_{N}}\left(x\right)}\right|+\left|\frac{K_{N}-1}{N}\frac{\hat{\gamma}_{K_{N}}\left(x\right)}{1-\hat{\gamma}_{K_{N}}\left(x\right)}-\frac{\alpha\,\hat{\zeta}_{K_{N}}\left(x\right)}{1-\hat{\zeta}_{K_{N}}\left(x\right)}\right|\right]\\
 & =\frac{1}{\alpha}\left|\frac{\alpha\,\hat{\gamma}_{K_{N}}\left(x\right)}{1-\hat{\gamma}_{K_{N}}\left(x\right)}-\frac{K_{N}-1}{N}\frac{\hat{\gamma}_{K_{N}}\left(x\right)}{1-\hat{\gamma}_{K_{N}}\left(x\right)}\right|\\
 & =\frac{1}{\alpha}\left|\alpha-\frac{K_{N}-1}{N}\right|\left|\frac{\hat{\gamma}_{K_{N}}\left(x\right)}{1-\hat{\gamma}_{K_{N}}\left(x\right)}\right|\\
 & =\frac{1}{\alpha}\left|\alpha-\frac{K_{N}}{N}+\frac{1}{N}\right|\left|\frac{\hat{\gamma}_{K_{N}}\left(x\right)}{1-\hat{\gamma}_{K_{N}}\left(x\right)}\right|.
\end{align*}
We set $f:\left[-\frac{b}{1+b},\frac{\beta_{1}}{1-\beta_{1}}\right]\rightarrow\IR$,
$f\left(t\right)\coloneq\frac{t}{t+1}$, and we fix some $t_{0}\in\left(-\frac{b}{1+b},\frac{\beta_{1}}{1-\beta_{1}}\right)$.
Using a Taylor expansion, we see that
\[
f\left(t\right)=f\left(t_{0}\right)+f'\left(\tau\right)\left(t-t_{0}\right)
\]
 holds for some $\tau$ which lies between $t$ and $t_{0}$. Next
we upper bound the derivative
\[
\left|f'\left(\tau\right)\right|=\frac{1}{\left(\tau+1\right)^{2}}\leq\frac{1}{\left(-\frac{b}{1+b}+1\right)^{2}}=\left(1+b\right)^{2},\quad\tau\in\left(-\frac{b}{1+b},\frac{\beta_{1}}{1-\beta_{1}}\right).
\]
Hence,
\begin{align*}
\left|\hat{\gamma}_{K_{N}}\left(x\right)-\hat{\zeta}_{K_{N}}\left(x\right)\right| & =\left|f\left(\frac{\hat{\gamma}_{K_{N}}\left(x\right)}{1-\hat{\gamma}_{K_{N}}\left(x\right)}\right)-f\left(\frac{\hat{\zeta}_{K_{N}}\left(x\right)}{1-\hat{\zeta}_{K_{N}}\left(x\right)}\right)\right|\\
 & \leq\sup_{\tau\in\left(-\frac{b}{1+b},\frac{\beta_{1}}{1-\beta_{1}}\right)}\left|f'\left(\tau\right)\right|\left|\frac{\hat{\gamma}_{K_{N}}\left(x\right)}{1-\hat{\gamma}_{K_{N}}\left(x\right)}-\frac{\hat{\zeta}_{K_{N}}\left(x\right)}{1-\hat{\zeta}_{K_{N}}\left(x\right)}\right|\\
 & \leq\frac{1}{\alpha}\left|\alpha-\frac{K_{N}}{N}+\frac{1}{N}\right|\left|\frac{\hat{\gamma}_{K_{N}}\left(x\right)}{1-\hat{\gamma}_{K_{N}}\left(x\right)}\right|\sup_{\tau\in\left(-\frac{b}{1+b},\frac{\beta_{1}}{1-\beta_{1}}\right)}\left|f'\left(\tau\right)\right|\\
 & \leq\frac{1}{\alpha}\left(1+b\right)^{2}\frac{\beta_{1}}{1-\beta_{1}}\left|\alpha-\frac{K_{N}}{N}+\frac{1}{N}\right|.
\end{align*}
Let $\beta\in I_{l}$ and fix $b>\beta$. Let $n\in\IN$ and $x\in L_{N,n}$.
We first show
\[
\left|m\left(\hat{\gamma}_{K_{N}}\left(x\right)\right)^{2}-m\left(\hat{\zeta}_{K_{N}}\left(x\right)\right)^{2}\right|\leq\frac{2}{K_{N}-1}.
\]
By Definition \ref{def:pair_K},
\[
m\left(\hat{\gamma}_{K_{N}}\left(x\right)\right)^{2}=\frac{1}{n}\sum_{t=1}^{n}\frac{1}{K_{N}\left(K_{N}-1\right)}\sum_{1\leq i,j\leq K_{N};\,i\neq j}x_{i}^{(t)}x_{j}^{(t)},
\]
and by Definition \ref{def:zeta^hat},
\[
m\left(\hat{\zeta}_{K_{N}}\left(x\right)\right)^{2}=\frac{1}{n}\sum_{t=1}^{n}\frac{1}{K_{N}^{2}}\left(\sum_{i=1}^{K_{N}}x_{i}^{(t)}\right)^{2}.
\]
So
\begin{align*}
m\left(\hat{\gamma}_{K_{N}}\left(x\right)\right)^{2} & =\frac{1}{n}\sum_{t=1}^{n}\frac{1}{K_{N}\left(K_{N}-1\right)}\sum_{1\leq i,j\leq K_{N};\,i\neq j}x_{i}^{(t)}x_{j}^{(t)}\\
 & =\frac{1}{n}\sum_{t=1}^{n}\frac{1}{K_{N}\left(K_{N}-1\right)}\left(K_{N}+\sum_{1\leq i,j\leq K_{N};\,i\neq j}x_{i}^{(t)}x_{j}^{(t)}\right)-\frac{1}{K_{N}-1}\\
 & =\frac{1}{n}\sum_{t=1}^{n}\frac{1}{K_{N}\left(K_{N}-1\right)}\left(\sum_{i=1}^{K_{N}}\left(x_{i}^{(t)}\right)^{2}+\sum_{1\leq i,j\leq K_{N};\,i\neq j}x_{i}^{(t)}x_{j}^{(t)}\right)-\frac{1}{K_{N}-1}\\
 & =\frac{1}{n}\sum_{t=1}^{n}\frac{1}{K_{N}\left(K_{N}-1\right)}\left(\sum_{i=1}^{K_{N}}x_{i}^{(t)}\right)^{2}-\frac{1}{K_{N}-1}\\
 & =\frac{K_{N}}{K_{N}-1}\frac{1}{n}\sum_{t=1}^{n}\frac{1}{K_{N}^{2}}\left(\sum_{i=1}^{K_{N}}x_{i}^{(t)}\right)^{2}-\frac{1}{K_{N}-1}\\
 & =\frac{K_{N}}{K_{N}-1}m\left(\hat{\zeta}_{K_{N}}\left(x\right)\right)^{2}-\frac{1}{K_{N}-1}\\
 & =\frac{1}{K_{N}-1}\left(K_{N}m\left(\hat{\zeta}_{K_{N}}\left(x\right)\right)^{2}-1\right).
\end{align*}
From this last display, we derive two statements. First,
\begin{align*}
m\left(\hat{\gamma}_{K_{N}}\left(x\right)\right)^{2} & \leq m\left(\hat{\zeta}_{K_{N}}\left(x\right)\right)^{2}\quad\iff\\
\frac{1}{K_{N}-1}\left(K_{N}m\left(\hat{\zeta}_{K_{N}}\left(x\right)\right)^{2}-1\right) & \leq m\left(\hat{\zeta}_{K_{N}}\left(x\right)\right)^{2}\quad\iff\\
K_{N}m\left(\hat{\zeta}_{K_{N}}\left(x\right)\right)^{2}-1 & \leq\left(K_{N}-1\right)m\left(\hat{\zeta}_{K_{N}}\left(x\right)\right)^{2}\quad\iff\\
m\left(\hat{\zeta}_{K_{N}}\left(x\right)\right)^{2} & \leq1.
\end{align*}
Since the function $\beta\mapsto m\left(\beta\right)^{2}$ is strictly
increasing per Lemma \ref{lem:m_beta_increasing}, this chain of equivalences
implies
\[
\hat{\gamma}_{K_{N}}\left(x\right)\leq\hat{\zeta}_{K_{N}}\left(x\right),
\]
which holds with equality if and only if $m\left(\hat{\zeta}_{K_{N}}\left(x\right)\right)=1$,
which itself is equivalent to $\hat{\zeta}_{K_{N}}\left(x\right)=\infty$
by Definition \ref{def:m_beta}. By Definition \ref{def:zeta^hat},
$\hat{\zeta}_{K_{N}}\left(x\right)=\infty$ holds if and only if the
sample $x$ is such that every vote is unanimous, i.e. for all $t\in\IN_{n}$
and all $i,j\in\IN_{N}$, $x_{i}^{(t)}=x_{j}^{(t)}$. In this special
case, we have $\hat{\gamma}_{K_{N}}\left(x\right)=\hat{\zeta}_{K_{N}}\left(x\right)$.
If there is even a single dissenting vote in the sample $x$, i.e.
there is some $t\in\IN_{n}$ and some $i,j\in\IN_{K_{N}}$ with $x_{i}^{(t)}\neq x_{j}^{(t)}$,
then $\hat{\gamma}_{K_{N}}\left(x\right)<\hat{\zeta}_{K_{N}}\left(x\right)$.
Since we are assuming $x\in L_{N,n}$, $\hat{\gamma}_{K_{N}}\left(x\right)<\hat{\zeta}_{K_{N}}\left(x\right)$
is satisfied.

The second calculation yields
\begin{align*}
\left|m\left(\hat{\gamma}_{K_{N}}\left(x\right)\right)^{2}-m\left(\hat{\zeta}_{K_{N}}\left(x\right)\right)^{2}\right| & =\left|\frac{1}{K_{N}-1}\left(K_{N}\,m\left(\hat{\zeta}_{K_{N}}\left(x\right)\right)^{2}-1\right)-m\left(\hat{\zeta}_{K_{N}}\left(x\right)\right)^{2}\right|\\
 & \leq\frac{1}{K_{N}-1}\left(m\left(\hat{\zeta}_{K_{N}}\left(x\right)\right)^{2}+1\right)\\
 & \leq\frac{2}{K_{N}-1}.
\end{align*}
Note that this upper bound holds for any $x\in T_{K_{N}}^{-1}\left(J'_{l}\right)$
and not only those $x$ which also lie in $T_{K_{N}}^{-1}\left(K_{N}^{2}\left[m\left(\beta_{2}\right)^{2},m\left(b\right)^{2}\right]\right)$.

We write the Taylor expansion
\[
m^{-1}\left(\sqrt{y}\right)=m^{-1}\left(\sqrt{y_{0}}\right)+\left(m^{-1}\left(\sqrt{\upsilon}\right)\right)'\left(y-y_{0}\right)
\]
for fixed $y,y_{0}\in\left[m\left(\beta_{2}\right)^{2},\infty\right)$
and some $\upsilon$ which lies between $y$ and $y_{0}$. The derivative
of $y\mapsto m^{-1}\left(\sqrt{y}\right)$ is
\[
\left(m^{-1}\left(\sqrt{y}\right)\right)'=\frac{1}{m'\left(m^{-1}\left(\sqrt{y}\right)\right)}\frac{1}{2\sqrt{y}}.
\]
We upper bound this derivative considering $x\in L_{N,n}$. Under
this assumption, we have 
\[
m\left(\hat{\gamma}_{K_{N}}\left(x\right)\right)^{2}<m\left(\hat{\zeta}_{K_{N}}\left(x\right)\right)^{2}\leq m\left(b\right)^{2}.
\]
As $m'$ is strictly positive and continuous on the compact interval
$\left[\beta_{2},b\right]$, it reaches a strictly positive minimum,
and hence
\begin{align*}
\frac{1}{m'\left(m^{-1}\left(\sqrt{y}\right)\right)} & \leq C
\end{align*}
holds for some constant $C$ and for all $y\in\left[m\left(\beta_{2}\right)^{2},m\left(b\right)^{2}\right]$
. Therefore,
\begin{align*}
\left(m^{-1}\left(\sqrt{y}\right)\right)' & =\frac{1}{m'\left(m^{-1}\left(\sqrt{y}\right)\right)}\frac{1}{2\sqrt{y}}\\
 & \leq C\frac{1}{2m\left(\beta_{2}\right)}
\end{align*}
holds for all $y\in\left[m\left(\beta_{2}\right)^{2},m\left(b\right)^{2}\right]$,
and, in particular, we have
\begin{align*}
\left|\hat{\gamma}_{K_{N}}\left(x\right)-\hat{\zeta}_{K_{N}}\left(x\right)\right| & =\left|m^{-1}\left(\sqrt{m\left(\hat{\gamma}_{K_{N}}\left(x\right)\right)^{2}}\right)-m^{-1}\left(\sqrt{m\left(\hat{\zeta}_{K_{N}}\left(x\right)\right)^{2}}\right)\right|\\
 & \leq C\frac{1}{2m\left(\beta_{2}\right)}\left|m\left(\hat{\gamma}_{K_{N}}\left(x\right)\right)^{2}-m\left(\hat{\zeta}_{K_{N}}\left(x\right)\right)^{2}\right|\\
 & \leq C\frac{1}{2m\left(\beta_{2}\right)}\frac{2}{K_{N}-1},
\end{align*}
which yields the claim.
\end{proof}

\appendix

\section*{Appendix}

We present a number of concepts and auxiliary results we use. Some
of these are proved in the first two articles.
\begin{defn}
\label{def:m_beta}Let $\beta\geq0$. The equation
\begin{equation}
\tanh\left(\beta x\right)=x,\quad x\in\IR,\label{eq:CW}
\end{equation}
is called the Curie-Weiss equation. We define $m\left(\beta\right)$
to be the largest solution to (\ref{eq:CW}). In order to obtain a
function $m:\left[0,\infty\right]\rightarrow\left[0,1\right]$, we
set $m\left(\infty\right)\coloneq1$.
\end{defn}

Next we give two lemmas concerning properties of the function $m$.
\begin{lem}
\label{lem:m_beta_increasing}The mapping $m:\left[1,\infty\right)\rightarrow\left[0,1\right]$
is strictly increasing and $\lim_{\beta\rightarrow\infty}m\left(\beta\right)=1$.
\end{lem}

\begin{lem}
\label{lem:m_beta_diff}The mapping $m:\left(1,\infty\right)\rightarrow\left[0,1\right]$
is continuously differentiable.
\end{lem}

\begin{prop}
\label{prop:appr_correlations}For all $\beta\in\IR,\beta\neq1$,
and all $k\in\IN$ , the correlation $\IE_{\beta,N}X_{1}\cdots X_{k}$
is equal to $0$ for all $k$ odd and all $N\in\IN$. For $k$ even,
$\IE_{\beta,N}X_{1}\cdots X_{k}$ is asymptotically equal to
\[
\IE_{\beta,N}X_{1}\cdots X_{k}\approx\begin{cases}
\left(k-1\right)!!\left(\frac{\beta}{1-\beta}\right)^{\frac{k}{2}}\,\frac{1}{N^{\frac{k}{2}}} & \text{if }\,\beta<1,\\
m\left(\beta\right)^{k} & \text{if }\,\beta>1,
\end{cases}
\]
where $0<m\left(\beta\right)<1$ for $\beta>1$ is the constant from
Definition \ref{def:m_beta}.

Let $k$ be even. There are constants $\boldsymbol{C}_{\textup{high}},\boldsymbol{C}_{\textup{low}}>0$
with the following property. For all $0<\beta<1$, the bound 
\begin{equation}
\left|\IE_{\beta,N}X_{1}\cdots X_{k}-\left(k-1\right)!!\left(\frac{\beta}{1-\beta}\right)^{\frac{k}{2}}\,\frac{1}{N^{\frac{k}{2}}}\right|<\boldsymbol{C}_{\textup{high}}\left(\frac{\ln N}{N}\right)^{\frac{k+2}{2}}\label{eq:corr_UB_h}
\end{equation}
holds for all $N\in\IN$. For all $\beta>1$, the bound 
\begin{equation}
\left|\IE_{\beta,N}X_{1}\cdots X_{k}-m\left(\beta\right)^{k}\right|<\boldsymbol{C}_{\textup{low}}\frac{\left(\ln N\right)^{\frac{3}{2}}}{\sqrt{N}}\label{eq:corr_UB_l}
\end{equation}
holds for all $N\in\IN$.
\end{prop}

\begin{proof}
This is Proposition 20 in \cite{BalNauTo2025}.
\end{proof}
\begin{thm}[Slutsky]
\label{thm:slutsky}Let $\left(Y_{n}\right)_{n\in\IN}$ and $\left(Z_{n}\right)_{n\in\IN}$
be sequences of random variables, $Y$ a random variable, and $a\in\IR$
a constant such that $Y_{n}\xrightarrow[n\rightarrow\infty]{\textup{d}}Y$
and $Z_{n}\xrightarrow[n\rightarrow\infty]{\textup{p}}a$. Then
\[
Y_{n}+Z_{n}\xrightarrow[n\rightarrow\infty]{\textup{d}}Y+a\quad\text{and}\quad Y_{n}Z_{n}\xrightarrow[n\rightarrow\infty]{\textup{d}}aY.
\]
\end{thm}

\begin{thm}[Continuous Mapping]
\label{thm:cont_mapping}Let $\left(Y_{n}\right)_{n\in\IN}$ be a
sequence of random variables and $Y$ a random variable, each of them
taking values in some subset $A\subset\IR$, such that $Y_{n}\xrightarrow[n\rightarrow\infty]{\textup{p}}Y$,
and let $g:A\rightarrow\IR$ be a continuous function. Then
\[
g\left(Y_{n}\right)\xrightarrow[n\rightarrow\infty]{\textup{p}}g\left(Y\right).
\]
\end{thm}

\begin{lem}
\label{lem:conv_restr_sequence}Let $\left(Y_{n}\right)_{n\in\IN}$
be a sequence of random variables and $\left(M_{n}\right)_{n\in\IN}$
a sequence of positive numbers such that
\[
\left|Y_{n}\right|\leq M_{n},\quad n\in\IN,
\]
is satisfied. Let $\nu$ be a probability measure on $\IR$, and assume
the convergence $Y_{n}\xrightarrow[n\rightarrow\infty]{\textup{d}}\nu$.
Finally, let $\left(B_{n}\right)_{n\in\IN}$ be a sequence of measurable
sets which satisfies
\[
\IP\left\{ Y_{n}\in B_{n}\right\} =o\left(\frac{1}{M_{n}}\right).
\]
Then we have
\[
\II_{\left\{ Y_{n}\in B_{n}^{c}\right\} }Y_{n}\xrightarrow[n\rightarrow\infty]{\textup{d}}\nu.
\]
\end{lem}

\begin{thm}[Delta Method]
\label{thm:delta_method}Let $\left(Y_{n}\right)_{n\in\IN}$ be a
sequence of random variables such that $\IE Y_{n}=\mu\in\IR$ for
all $n\in\IN$ and $\sqrt{n}\left(Y_{n}-\mu\right)\xrightarrow[n\rightarrow\infty]{\textup{d}}\mathcal{N}\left(0,\sigma^{2}\right)$
for a constant $\sigma>0$. Let $f:D\rightarrow\IR$ be a continuously
differentiable function with domain $D\subset\IR$ such that $Y_{n}\in D$
for all $n\in\IN$. Assume $f'\left(\mu\right)\neq0$. Then
\[
\sqrt{n}\left(f\left(Y_{n}\right)-f\left(\mu\right)\right)\xrightarrow[n\rightarrow\infty]{\textup{d}}\mathcal{N}\left(0,\left(f'\left(\mu\right)\right)^{2}\sigma^{2}\right)
\]
is satisfied.
\end{thm}

Recall Notation \ref{notation:infty} for the expressions $\left[-\infty,\infty\right]$
and $\left[0,\infty\right]$.
\begin{defn}
\label{def:LDP}Let $\left(P_{n}\right)_{n\in\IN}$ be a sequence
of probability measures on a metric space $\mathcal{X}$, let $\left(a_{n}\right)_{n\in\IN}$
be a sequence of positive numbers with $a_{n}\xrightarrow[n\rightarrow\infty]{}\infty$,
and let $I:\mathcal{X}\rightarrow\left[0,\infty\right]$ be a function.
If $I$ is lower semi-continuous, i.e. its level sets $\left\{ x\in\mathcal{X}\,|\,I\left(x\right)\leq\alpha\right\} $
are closed for each $\alpha\in\left[0,\infty\right)$, we call $I$
a rate function. If the level sets are compact in $\mathcal{X}$ for
each $\alpha\in\left[0,\infty\right)$, we call $I$ a good rate function.
If $I$ is a good rate function, and the two conditions
\begin{enumerate}
\item $\limsup_{n\rightarrow\infty}\frac{1}{a_{n}}\ln P_{n}K\le-\inf_{x\in K}I\left(x\right)$
for each closed set $K\subset\mathcal{X}$,
\item $\liminf_{n\rightarrow\infty}\frac{1}{a_{n}}\ln P_{n}G\geq-\inf_{x\in G}I\left(x\right)$
for each open set $G\subset\mathcal{X}$
\end{enumerate}
hold, then we say that the sequence $\left(P_{n}\right)_{n\in\IN}$
satisfies a large deviations principle with rate $a_{n}$ and rate
function $I$. If $\left(Y_{n}\right)_{n\in\IN}$ is a sequence of
random variables taking values in $\mathcal{X}$ such that, for each
$n\in\IN$, $Y_{n}$ follows the distribution $P_{n}$, we will also
say that $\left(Y_{n}\right)_{n\in\IN}$ satisfies a large deviations
principle with rate $a_{n}$ and rate function $I$.

In our applications of large deviations principles, the metric space
$\mathcal{X}$ will be $\IR$ or $\left[-\infty,\infty\right]$.
\end{defn}

\begin{defn}
\label{def:exchange}Let $Y_{1},\ldots,Y_{n}$ be real random variables
with joint distribution $P$ on $\IR^{n}$. We say that $Y_{1},\ldots,Y_{n}$
are exchangeable if for all permutations $\pi$ on $\IN_{n}$ the
random vector $\left(Y_{\pi\left(1\right)},\ldots,Y_{\pi\left(n\right)}\right)$
has joint distribution $P$.
\end{defn}

\begin{lem}
\label{lem:exchange}The random variables $X_{1},\ldots,X_{N}$ are
exchangeable under the distribution $\IP_{\beta,N}$ in Definition
\ref{def:CWM}.
\end{lem}

\begin{proof}
This is Lemma 59 in \cite{BalNauTo2025}.
\end{proof}
\bibliographystyle{plain}

\begin{thebibliography}{10}

\bibitem{AlBaKhBu2022}
Alexander~P. Alodjants, A.~Yu. Bazhenov, A.~Yu. Khrennikov, and A.~V.
  Bukhanovsky.
\newblock Mean-field theory of social laser.
\newblock {\em Sci. Rep.}, 12(8566), 2022.

\bibitem{BaKaKPRS1992}
Fran\c{c}ois Baccelli, Fridrikh~I. Karpelevich, M.~Ya. Kelbert, Anatolii~A.
  Puhalskii, A.~N. Rybko, and Yuri~M. Suhov.
\newblock A mean-field limit for a class of queueing networks.
\newblock {\em J. Stat. Phys.}, 66:803--825, 1992.

\bibitem{BaMePeTo2023}
Miguel Ballesteros, Rams\'es~H. Mena, Jos\'e~Luis P\'erez, and Gabor Toth.
\newblock Detection of an arbitrary number of communities in a block spin ising
  model.
\newblock 2023.
\newblock arXiv:2311.18112.

\bibitem{BaMeSiTo2025}
Miguel Ballesteros, Rams\'es~H. Mena, Arno Siri-J\'egousse, and Gabor Toth.
\newblock Reconstruction of the probability measure and the coupling parameters
  in a curie-weiss model.
\newblock 2025.
\newblock arXiv:2505.21778.

\bibitem{BalNauTo2025}
Miguel Ballesteros, Ivan Naumkin, and Gabor Toth.
\newblock Approximation techniques for the reconstruction of the probability
  measure and the coupling parameters in a {C}urie-{W}eiss model for large
  populations.
\newblock 2025.
\newblock arXiv:2507.17073.

\bibitem{BRS2019}
Quentin Berthet, Philippe Rigollet, and Piyush Srivastava.
\newblock Exact recovery in the ising blockmodel.
\newblock {\em Ann. Statist.}, 47(4):1805--1834, 2019.

\bibitem{BouMcDMu2007}
Jean-Yves~Le Boudec, David McDonald, and Jochen Mundinger.
\newblock A generic mean field convergence result for systems of interacting
  objects.
\newblock In {\em Fourth International Conference on the Quantitative
  Evaluation of Systems (QEST 2007)}, pages 3--18, 2007.

\bibitem{BD2001}
William~A. Brock and Steven~N. Durlauf.
\newblock Discrete choice with social interactions.
\newblock {\em Rev. Econ. Stud.}, 68(2):235--260, 2001.

\bibitem{DieKauKa2023}
Hung~T. Diep, Miron Kaufman, and Sanda Kaufman.
\newblock An agent-based statistical physics model for political polarization:
  A monte carlo study.
\newblock {\em Entropy}, 25(7), 2023.

\bibitem{Ising1925}
Ernst Ising.
\newblock Beitrag zur {T}heorie des {F}erromagnetismus.
\newblock {\em Zeitschrift f\"ur Physik}, 31:253--258, 1925.

\bibitem{JovaRose1988}
Boyan Jovanovic and Robert~W. Rosenthal.
\newblock Anonymous sequential games.
\newblock {\em J. Math. Econ.}, 17(1):77--87, 1988.

\bibitem{KT2021c}
Werner Kirsch and Gabor Toth.
\newblock Optimal weights in a two-tier voting system with mean-field voters.
\newblock arXiv:2111.08636, November 2021.

\bibitem{LoweSchu2020}
Matthias L\"owe and Kristina Schubert.
\newblock Exact recovery in block spin {Ising} models at the critical line.
\newblock {\em Electron. J. Stat.}, 14:1796--1815, 2020.

\bibitem{Onsager1944}
Lars Onsager.
\newblock Crystal statistics {I}. {A} two-dimensional model with an
  order-disorder transition.
\newblock {\em Phys. Rev.}, 65(3-4):117--149, 1944.

\end{thebibliography}

\end{document}